\documentclass[10pt, a4paper]{amsart}

\title[TL categories with labelled regions and JW projectors]{Temperley--Lieb categories with coloured regions and Jones--Wenzl projectors}

\date{March 30, 2026}
\author{Cameron Howat}

\author{Robert Laugwitz}
\address{School of Mathematical Sciences,
University of Nottingham, Nottingham, NG7 2RD, United Kingdom}
\email{robert.laugwitz@nottingham.ac.uk}

\author{Martin Ray}
\email{pmymr5@nottingham.ac.uk}

\usepackage{amsmath}
\usepackage{amsfonts}
\usepackage{amsthm}
\usepackage{mathtools}
\usepackage{amssymb}
\usepackage[nobysame,initials]{amsrefs}
\usepackage[english]{babel}
\usepackage{url}
\usepackage{graphicx}
\usepackage[
colorlinks=true,
linkcolor=black, 
anchorcolor=black,
citecolor=black, 
urlcolor=black, 
]{hyperref}
\usepackage{cleveref}
\usepackage{tikz}
\usepackage{mathdots}
\usepackage{mathtools}
\usepackage{faktor}
\usepackage{bbm}

\numberwithin{equation}{section}

\newtheorem{theorem}{Theorem}[section]

\newtheorem{proposition}[theorem]{Proposition}
\newtheorem{corollary}[theorem]{Corollary}
\newtheorem{lemma}[theorem]{Lemma}

\newtheorem{theorem*}{Theorem}

\theoremstyle{definition}
\newtheorem{definition}[theorem]{Definition}

\newtheorem{assumption}[theorem]{Assumption}

\theoremstyle{remark}
\newtheorem{example}[theorem]{Example}
\newtheorem{remark}[theorem]{Remark}

\newcommand{\cC}{\mathcal{C}}
\newcommand{\cD}{\mathcal{D}}

\newcommand{\cI}{\mathcal{I}}

\newcommand{\proj}[1]{\mathrm{proj}\text{-}{#1}}

\newcommand{\bi}{\mathbf{i}}
\newcommand{\bj}{\mathbf{j}}

\def\id{\mathrm{id}}

\newcommand{\NN}{\mathbb{N}}
\newcommand{\ZZ}{\mathbb{Z}}

\newcommand{\KK}{\Bbbk}
\newcommand{\Hom}{\mathrm{Hom}}
\newcommand{\TL}{\mathrm{TL}}
\newcommand{\End}{\mathrm{End}}
\newcommand{\Kar}{\mathrm{Kar}}
\newcommand{\Indec}{\mathrm{Indec}}

\newcommand{\tr}{\operatorname{tr}}
\newcommand{\one}{\mathbbm{1}}

\begin{document}
\newcommand\MatchingBoxT{%
  \begin{tikzpicture}[scale = 0.7]
    \foreach \x in {1,3,4}{
       \draw[circle,fill] (\x,0)circle[radius=0.5mm]node[below]{};
       \draw[circle,fill] (\x,4)circle[radius=0.5mm]node[below]{};
       }
   \draw[line width=0.25mm](1,0) -- (1,1);
   \draw[line width=0.25mm](3,0) -- (3,1);
   \draw[line width=0.25mm](1,4) -- (1,3);
   \draw[line width=0.25mm](3,4) -- (3,3);
   \draw[line width=0.25mm](4,0) -- (4,4);
   \draw[line width=0.4mm](0.75,1) -- (3.25,1);
   \draw[line width=0.4mm](0.75,3) -- (3.25,3);
   \draw[line width=0.4mm](0.75,1) -- (0.75,3);
   \draw[line width=0.4mm](3.25,1) -- (3.25,3);
   \draw[circle,fill](2,2.75)circle[radius=0mm]node[below]{  $f_{n-1}$};
   \draw[circle,fill](2,1)circle[radius=0mm]node[below]{  $\dots$};
   \draw[circle,fill](2,4)circle[radius=0mm]node[below]{  $\dots$};

  \end{tikzpicture}%
}

\newcommand\MatchingBoxL{%
  \begin{tikzpicture}[scale = 0.7]
    \foreach \x in {1,3,4}{
       \draw[circle,fill] (\x+7,-0.3)circle[radius=0.5mm]node[below]{};
       \draw[circle,fill] (\x+7,4.3)circle[radius=0.5mm]node[below]{};
       }
    
   \draw[line width=0.25mm](8,-0.3) -- (8,0.25);
   \draw[line width=0.25mm](10,-0.3) -- (10,0.25);
   \draw[line width=0.25mm](8,4.3) -- (8,3.75);
   \draw[line width=0.25mm](10,4.3) -- (10,3.75);
   \draw[line width=0.4mm](7.75,0.25) -- (10.25,0.25);
   \draw[line width=0.4mm](7.75,1.25) -- (10.25,1.25);
   \draw[line width=0.4mm](7.75,0.25) -- (7.75,1.25);
   \draw[line width=0.4mm](10.25,0.25) -- (10.25,1.25);
   \draw[line width=0.4mm](7.75,2.75) -- (10.25,2.75);
   \draw[line width=0.4mm](7.75,3.75) -- (10.25,3.75);
   \draw[line width=0.4mm](7.75,2.75) -- (7.75,3.75);
   \draw[line width=0.4mm](10.25,2.75) -- (10.25,3.75); 
   \draw[line width=0.25mm](8,1.25) -- (8,2.75);
   \draw[line width=0.25mm](9.25,1.25) -- (9.25,2.75);
   
   \draw[line width=0.25mm](11,-0.3) -- (11,1.5);
   \draw[line width=0.25mm](11,4.3) -- (11,2.5);
   \draw[line width=0.25mm](10,1.25) -- (10,1.5);
   \draw[line width=0.25mm](10,2.75) -- (10,2.5);
   \draw[line width=0.25mm](10,1.5) to[bend left=90] (11,1.5);
   \draw[line width=0.25mm](10,2.5) to[bend right=90] (11,2.5);
   
   \draw[circle,fill](9,0.5)circle[radius=0mm]node[below]{  $\dots$};
   \draw[circle,fill](9,4.5)circle[radius=0mm]node[below]{  $\dots$};
   \draw[circle,fill](8.625,2.5)circle[radius=0mm]node[below]{  $\dots$};
   \draw[circle,fill](9,4)circle[radius=0mm]node[below]{  $f_{n-1}$};
   \draw[circle,fill](9,1.5)circle[radius=0mm]node[below]{  $f_{n-1}$};
  \end{tikzpicture}%
}

\newcommand\MatchingBoxes{%
  \begin{tikzpicture}[scale=0.5]
    \foreach \x in {1,3,6,8}{
       \draw[circle,fill] (\x,0)circle[radius=0.5mm]node[below]{};
       \draw[circle,fill] (\x,3)circle[radius=0.5mm]node[below]{};
       \draw[line width=0.25mm](\x,3) -- (\x,2.5);
       \draw[line width=0.25mm](\x,0) -- (\x,0.5);
       }
    \draw[line width=0.25mm](5,0) -- (5,0.5);
    \draw[line width=0.25mm](4,0) -- (4,0.5);
	\draw[circle,fill] (5,0)circle[radius=0.5mm]node[below]{};
     \draw[circle,fill] (4,0)circle[radius=0.5mm]node[below]{};

   \draw[line width=0.4mm](0.5,0.5) -- (8.5,0.5);
   \draw[line width=0.4mm](0.5,0.5) -- (0.5,2.5);
   \draw[line width=0.4mm](8.5,0.5) -- (8.5,2.5);
   \draw[line width=0.4mm](0.5,2.5) -- (8.5,2.5);

   \draw[line width=0.25mm](4,2.75) to[bend left=90] (5,2.75);
   \draw[line width=0.25mm](4,2.75) -- (4,2.5);
   \draw[line width=0.25mm](5,2.75) -- (5,2.5);

    \draw[circle,fill](2,0.9)circle[radius=0mm]node[below]{  $\dots$};
    \draw[circle,fill](7,3.75)circle[radius=0mm]node[below]{  $\dots$};
	\draw[circle,fill](7,0.9)circle[radius=0mm]node[below]{  $\dots$};
    \draw[circle,fill](2,3.75)circle[radius=0mm]node[below]{  $\dots$};

    \draw[circle,fill](4.5,2.25)circle[radius=0mm]node[below]{  $f_{n}$};
 
    \foreach \x in {11,13,16,18}{
       \draw[circle,fill] (\x,0)circle[radius=0.5mm]node[below]{};
       \draw[circle,fill] (\x,3)circle[radius=0.5mm]node[below]{};
       \draw[line width=0.25mm](\x,3) -- (\x,2.5);
       \draw[line width=0.25mm](\x,0) -- (\x,0.5);
       }
    \draw[line width=0.25mm](15,3) -- (15,2.5);
    \draw[line width=0.25mm](14,3) -- (14,2.5);
	\draw[circle,fill] (15,3)circle[radius=0.5mm]node[below]{};
    \draw[circle,fill] (14,3)circle[radius=0.5mm]node[below]{};
    
   \draw[line width=0.4mm](10.5,0.5) -- (18.5,0.5);
   \draw[line width=0.4mm](10.5,0.5) -- (10.5,2.5);
   \draw[line width=0.4mm](18.5,0.5) -- (18.5,2.5);
   \draw[line width=0.4mm](10.5,2.5) -- (18.5,2.5);
    
   \draw[line width=0.25mm](14,0.25) to[bend right=90] (15,0.25);
    \draw[line width=0.25mm](14,0.25) -- (14,0.5);
    \draw[line width=0.25mm](15,0.25) -- (15,0.5);
   
    \draw[circle,fill](12,0.9)circle[radius=0mm]node[below]{  $\dots$};
    \draw[circle,fill](17,3.75)circle[radius=0mm]node[below]{  $\dots$};
	\draw[circle,fill](17,0.9)circle[radius=0mm]node[below]{  $\dots$};
    \draw[circle,fill](12,3.75)circle[radius=0mm]node[below]{  $\dots$};

    \draw[circle,fill](14.5,2.25)circle[radius=0mm]node[below]{  $f_{n}$};
   
    \draw[circle,fill](9.5,2.125)circle[radius=0mm]node[below]{  $=$};
   
    \draw[circle,fill](19.5,2.25)circle[radius=0mm]node[below]{  $= 0$};
   
  \end{tikzpicture}%
}

\newcommand\MatchingBoxLOT{%
  \begin{tikzpicture}[scale = 0.75]
    \foreach \x in {1,3,4}{
       \draw[circle,fill] (\x,0)circle[radius=0.5mm]node[below]{};
       \draw[circle,fill] (\x,4)circle[radius=0.5mm]node[below]{};
       }
   \draw[line width=0.25mm](1,0) -- (1,1);
   \draw[line width=0.25mm](3,0) -- (3,1);
   \draw[line width=0.25mm](1,4) -- (1,3);
   \draw[line width=0.25mm](3,4) -- (3,3);
   \draw[line width=0.25mm](4,0) -- (4,4);
   \draw[line width=0.4mm](0.75,1) -- (3.25,1);
   \draw[line width=0.4mm](0.75,3) -- (3.25,3);
   \draw[line width=0.4mm](0.75,1) -- (0.75,3);
   \draw[line width=0.4mm](3.25,1) -- (3.25,3);
   \draw[circle,fill](2,2.75)circle[radius=0mm]node[below]{  $f_{n-1}^{ij}$};
   \draw[circle,fill](2,1)circle[radius=0mm]node[below]{  $\dots$};
   \draw[circle,fill](2,4)circle[radius=0mm]node[below]{  $\dots$};
   
\draw[circle,fill](0.5,2.35)circle[radius=0mm]node[below]{  $i$};
\draw[circle,fill](3.6,2.35)circle[radius=0mm]node[below]{  $i$};
\draw[circle,fill](4.3,2.35)circle[radius=0mm]node[below]{  $j$};
  
  \end{tikzpicture}%
}

\newcommand\MatchingBoxLOL{%
  \begin{tikzpicture}[scale = 0.75]
    \foreach \x in {1,3,4}{
     
       \draw[circle,fill] (\x+7,0)circle[radius=0.5mm]node[below]{};
       \draw[circle,fill] (\x+7,4)circle[radius=0.5mm]node[below]{};
       }
   \draw[line width=0.25mm](8,0) -- (8,0.25);
   \draw[line width=0.25mm](10,0) -- (10,0.25);
   \draw[line width=0.25mm](8,4) -- (8,3.75);
   \draw[line width=0.25mm](10,4) -- (10,3.75);
   \draw[line width=0.4mm](7.75,0.25) -- (10.25,0.25);
   \draw[line width=0.4mm](7.75,1.25) -- (10.25,1.25);
   \draw[line width=0.4mm](7.75,0.25) -- (7.75,1.25);
   \draw[line width=0.4mm](10.25,0.25) -- (10.25,1.25);
   \draw[line width=0.4mm](7.75,2.75) -- (10.25,2.75);
   \draw[line width=0.4mm](7.75,3.75) -- (10.25,3.75);
   \draw[line width=0.4mm](7.75,2.75) -- (7.75,3.75);
   \draw[line width=0.4mm](10.25,2.75) -- (10.25,3.75); 
   \draw[line width=0.25mm](8,1.25) -- (8,2.75);
   \draw[line width=0.25mm](9.25,1.25) -- (9.25,2.75);
   
   \draw[line width=0.25mm](11,0) -- (11,1.5);
   \draw[line width=0.25mm](11,4) -- (11,2.5);
   \draw[line width=0.25mm](10,1.25) -- (10,1.5);
   \draw[line width=0.25mm](10,2.75) -- (10,2.5);
   \draw[line width=0.25mm](10,1.5) to[bend left=90] (11,1.5);
   \draw[line width=0.25mm](10,2.5) to[bend right=90] (11,2.5);
   
   \draw[circle,fill](9,0.5)circle[radius=0mm]node[below]{  $\dots$};
   \draw[circle,fill](9,4.5)circle[radius=0mm]node[below]{  $\dots$};
   \draw[circle,fill](8.625,2.5)circle[radius=0mm]node[below]{  $\dots$};
   \draw[circle,fill](9,4)circle[radius=0mm]node[below]{  $f_{n-1}^{ij}$};
   \draw[circle,fill](9,1.5)circle[radius=0mm]node[below]{  $f_{n-1}^{ij}$};

\draw[circle,fill](7.65,2.35)circle[radius=0mm]node[below]{  $i$};
\draw[circle,fill](9.8,2.35)circle[radius=0mm]node[below]{  $j$};
\draw[circle,fill](10.65,3.5)circle[radius=0mm]node[below]{  $i$};
\draw[circle,fill](10.65,1)circle[radius=0mm]node[below]{  $i$};
   
  \end{tikzpicture}%
}

\newcommand\MatchingBoxLET{%
  \begin{tikzpicture}[scale =0.75]
    \foreach \x in {1,3,4}{
       \draw[circle,fill] (\x,0)circle[radius=0.5mm]node[below]{};
       \draw[circle,fill] (\x,4)circle[radius=0.5mm]node[below]{};
       }
   \draw[line width=0.25mm](1,0) -- (1,1);
   \draw[line width=0.25mm](3,0) -- (3,1);
   \draw[line width=0.25mm](1,4) -- (1,3);
   \draw[line width=0.25mm](3,4) -- (3,3);
   \draw[line width=0.25mm](4,0) -- (4,4);
   \draw[line width=0.4mm](0.75,1) -- (3.25,1);
   \draw[line width=0.4mm](0.75,3) -- (3.25,3);
   \draw[line width=0.4mm](0.75,1) -- (0.75,3);
   \draw[line width=0.4mm](3.25,1) -- (3.25,3);
   \draw[circle,fill](2,2.75)circle[radius=0mm]node[below]{  $f_{n-1}^{ij}$};
   \draw[circle,fill](2,1)circle[radius=0mm]node[below]{  $\dots$};
   \draw[circle,fill](2,4)circle[radius=0mm]node[below]{  $\dots$};
  
\draw[circle,fill](0.5,2.35)circle[radius=0mm]node[below]{  $i$};
\draw[circle,fill](3.6,2.35)circle[radius=0mm]node[below]{  $j$};
\draw[circle,fill](4.3,2.35)circle[radius=0mm]node[below]{  $i$};

  \end{tikzpicture}%
}

\newcommand\MatchingBoxLEL{%
  \begin{tikzpicture}[scale =0.75]
    \foreach \x in {1,3,4}{     
       \draw[circle,fill] (\x+7,0)circle[radius=0.5mm]node[below]{};
       \draw[circle,fill] (\x+7,4)circle[radius=0.5mm]node[below]{};
       }
   \draw[line width=0.25mm](8,0) -- (8,0.25);
   \draw[line width=0.25mm](10,0) -- (10,0.25);
   \draw[line width=0.25mm](8,4) -- (8,3.75);
   \draw[line width=0.25mm](10,4) -- (10,3.75);
   \draw[line width=0.4mm](7.75,0.25) -- (10.25,0.25);
   \draw[line width=0.4mm](7.75,1.25) -- (10.25,1.25);
   \draw[line width=0.4mm](7.75,0.25) -- (7.75,1.25);
   \draw[line width=0.4mm](10.25,0.25) -- (10.25,1.25);
   \draw[line width=0.4mm](7.75,2.75) -- (10.25,2.75);
   \draw[line width=0.4mm](7.75,3.75) -- (10.25,3.75);
   \draw[line width=0.4mm](7.75,2.75) -- (7.75,3.75);
   \draw[line width=0.4mm](10.25,2.75) -- (10.25,3.75); 
   \draw[line width=0.25mm](8,1.25) -- (8,2.75);
   \draw[line width=0.25mm](9.25,1.25) -- (9.25,2.75);
   
   \draw[line width=0.25mm](11,0) -- (11,1.5);
   \draw[line width=0.25mm](11,4) -- (11,2.5);
   \draw[line width=0.25mm](10,1.25) -- (10,1.5);
   \draw[line width=0.25mm](10,2.75) -- (10,2.5);
   \draw[line width=0.25mm](10,1.5) to[bend left=90] (11,1.5);
   \draw[line width=0.25mm](10,2.5) to[bend right=90] (11,2.5);
   
   \draw[circle,fill](9,0.5)circle[radius=0mm]node[below]{  $\dots$};
   \draw[circle,fill](9,4.5)circle[radius=0mm]node[below]{  $\dots$};
   \draw[circle,fill](8.625,2.5)circle[radius=0mm]node[below]{  $\dots$};
   \draw[circle,fill](9,4)circle[radius=0mm]node[below]{  $f_{n-1}^{ij}$};
   \draw[circle,fill](9,1.5)circle[radius=0mm]node[below]{  $f_{n-1}^{ij}$};

\draw[circle,fill](7.65,2.35)circle[radius=0mm]node[below]{  $i$};
\draw[circle,fill](9.8,2.35)circle[radius=0mm]node[below]{  $i$};
\draw[circle,fill](10.65,3.5)circle[radius=0mm]node[below]{  $j$};
\draw[circle,fill](10.65,1)circle[radius=0mm]node[below]{  $j$};
      
  \end{tikzpicture}%
}

\newcommand\MatchingBoxLOUCapT{%
  \begin{tikzpicture}[scale = 0.75]
    \foreach \x in {1,3,4}{
       \draw[circle,fill] (\x,0)circle[radius=0.5mm]node[below]{};
       }
       \draw[circle,fill] (1,4)circle[radius=0.5mm]node[below]{};
   \draw[line width=0.25mm](1,0) -- (1,1);
   \draw[line width=0.25mm](3,0) -- (3,1);
   \draw[line width=0.25mm](1,4) -- (1,3);
   \draw[line width=0.25mm](3,4) -- (3,3);
   \draw[line width=0.25mm](4,0) -- (4,4);
   \draw[line width=0.4mm](0.75,1) -- (3.25,1);
   \draw[line width=0.4mm](0.75,3) -- (3.25,3);
   \draw[line width=0.4mm](0.75,1) -- (0.75,3);
   \draw[line width=0.4mm](3.25,1) -- (3.25,3);
   \draw[circle,fill](2,2.75)circle[radius=0mm]node[below]{ $f_{n-1}^{ij}$};
   \draw[circle,fill](2,1)circle[radius=0mm]node[below]{ $\dots$};
   \draw[circle,fill](2,4)circle[radius=0mm]node[below]{  $\dots$};

\draw[circle,fill](0.5,2.35)circle[radius=0mm]node[below]{  $i$};
\draw[circle,fill](3.6,2.35)circle[radius=0mm]node[below]{  $i$};
\draw[circle,fill](4.3,2.35)circle[radius=0mm]node[below]{  $j$};
   
    \draw[line width=0.25mm](3,4) to[bend left=90] (4,4);
  \end{tikzpicture}%
}

\newcommand\MatchingBoxLOUCapL{%
  \begin{tikzpicture}[scale = 0.75]
    \foreach \x in {1,3,4}{
       \draw[circle,fill] (\x+7,0)circle[radius=0.5mm]node[below]{};
       }
       \draw[circle,fill] (1+7,4)circle[radius=0.5mm]node[below]{};
   \draw[line width=0.25mm](8,0) -- (8,0.25);
   \draw[line width=0.25mm](10,0) -- (10,0.25);
   \draw[line width=0.25mm](8,4) -- (8,3.75);
   \draw[line width=0.25mm](10,4) -- (10,3.75);
   \draw[line width=0.4mm](7.75,0.25) -- (10.25,0.25);
   \draw[line width=0.4mm](7.75,1.25) -- (10.25,1.25);
   \draw[line width=0.4mm](7.75,0.25) -- (7.75,1.25);
   \draw[line width=0.4mm](10.25,0.25) -- (10.25,1.25);
   \draw[line width=0.4mm](7.75,2.75) -- (10.25,2.75);
   \draw[line width=0.4mm](7.75,3.75) -- (10.25,3.75);
   \draw[line width=0.4mm](7.75,2.75) -- (7.75,3.75);
   \draw[line width=0.4mm](10.25,2.75) -- (10.25,3.75); 
   \draw[line width=0.25mm](8,1.25) -- (8,2.75);
   \draw[line width=0.25mm](9.25,1.25) -- (9.25,2.75);
   
   \draw[line width=0.25mm](11,0) -- (11,1.5);
   \draw[line width=0.25mm](11,4) -- (11,2.5);
   \draw[line width=0.25mm](10,1.25) -- (10,1.5);
   \draw[line width=0.25mm](10,2.75) -- (10,2.5);
   \draw[line width=0.25mm](10,1.5) to[bend left=90] (11,1.5);
   \draw[line width=0.25mm](10,2.5) to[bend right=90] (11,2.5);
   
   \draw[circle,fill](9,0.5)circle[radius=0mm]node[below]{  $\dots$};
   \draw[circle,fill](9,4.5)circle[radius=0mm]node[below]{  $\dots$};
   \draw[circle,fill](8.625,2.5)circle[radius=0mm]node[below]{  $\dots$};
   \draw[circle,fill](9,4)circle[radius=0mm]node[below]{  $f_{n-1}^{ij}$};
   \draw[circle,fill](9,1.5)circle[radius=0mm]node[below]{  $f_{n-1}^{ij}$};

\draw[circle,fill](7.65,2.35)circle[radius=0mm]node[below]{  $i$};
\draw[circle,fill](9.8,2.35)circle[radius=0mm]node[below]{  $j$};
\draw[circle,fill](10.65,3.5)circle[radius=0mm]node[below]{  $i$};
\draw[circle,fill](10.65,1)circle[radius=0mm]node[below]{  $i$};
   
    \draw[line width=0.25mm](10,4) to[bend left=90] (11,4);
    
  \end{tikzpicture}%
}

\newcommand\MatchingBoxCapSA{%
  \begin{tikzpicture}[scale=0.7]
    \foreach \x in {1,3}{
       \draw[circle,fill] (\x,0)circle[radius=0.5mm]node[below]{};
       \draw[circle,fill] (\x,4)circle[radius=0.5mm]node[below]{};
       }

   \draw[line width=0.25mm](1,0) -- (1,1);
   \draw[line width=0.25mm](3,0) -- (3,1);
   \draw[line width=0.25mm](1,4) -- (1,3);
   \draw[line width=0.25mm](3,4) -- (3,3);
   \draw[line width=0.25mm](5,0) -- (5,4);
   \draw[line width=0.25mm](4,4) -- (4,3);
   \draw[line width=0.25mm](4,0) -- (4,1);
   \draw[line width=0.4mm](0.75,1) -- (4.25,1);
   \draw[line width=0.4mm](0.75,3) -- (4.25,3);
   \draw[line width=0.4mm](0.75,1) -- (0.75,3);
   \draw[line width=0.4mm](4.25,1) -- (4.25,3);
   \draw[circle,fill](2.5,2.8)circle[radius=0mm]node[below]{ $f_{n-1}^{ij}$};
   \draw[circle,fill](2,1)circle[radius=0mm]node[below]{ $\dots$};
   \draw[circle,fill](2,4)circle[radius=0mm]node[below]{ $\dots$};

    \draw[circle,fill](0.5,2.35)circle[radius=0mm]node[below]{ $i$};
    \draw[circle,fill](4.65,2.35)circle[radius=0mm]node[below]{ $i$};
    \draw[circle,fill](3.5,3.85)circle[radius=0mm]node[below]{ $j$};
    \draw[circle,fill](3.5,0.85)circle[radius=0mm]node[below]{ $j$};

    \draw[line width=0.25mm](4,4) to[bend left=90] (5,4);
    \draw[line width=0.25mm](4,0) to[bend right=90] (5,0);
  \end{tikzpicture}%
}

\newcommand\MatchingBoxCapSB{%
  \begin{tikzpicture}[scale=0.7]
    \foreach \x in {7,9}{
       \draw[circle,fill] (\x,0)circle[radius=0.5mm]node[below]{};
       \draw[circle,fill] (\x,4)circle[radius=0.5mm]node[below]{};
       }
   \draw[line width=0.25mm](7,0) -- (7,1);
   \draw[line width=0.25mm](9,0) -- (9,1);
   \draw[line width=0.25mm](7,4) -- (7,3);
   \draw[line width=0.25mm](9,4) -- (9,3);
   \draw[line width=0.25mm](11,0) -- (11,4);
   \draw[line width=0.25mm](10,0) -- (10,4);
   \draw[line width=0.4mm](6.75,1) -- (9.25,1);
   \draw[line width=0.4mm](6.75,3) -- (9.25,3);
   \draw[line width=0.4mm](6.75,1) -- (6.75,3);
   \draw[line width=0.4mm](9.25,1) -- (9.25,3);
   \draw[circle,fill](8,2.8)circle[radius=0mm]node[below]{ $f_{n-2}^{ij}$};
   \draw[circle,fill](8,1)circle[radius=0mm]node[below]{ $\dots$};
   \draw[circle,fill](8,4)circle[radius=0mm]node[below]{ $\dots$};

\draw[circle,fill](6.5,2.35)circle[radius=0mm]node[below]{ $i$};
\draw[circle,fill](10.5,2.35)circle[radius=0mm]node[below]{ $i$};
\draw[circle,fill](9.5,2.35)circle[radius=0mm]node[below]{ $j$};

    \draw[line width=0.25mm](10,4) to[bend left=90] (11,4);
    \draw[line width=0.25mm](10,0) to[bend right=90] (11,0);
  \end{tikzpicture}%
}

\newcommand\MatchingBoxCapSC{%
  \begin{tikzpicture}[scale=0.7]
    \foreach \x in {14,16}{
       \draw[circle,fill] (\x,0)circle[radius=0.5mm]node[below]{};
       \draw[circle,fill] (\x,4)circle[radius=0.5mm]node[below]{};
       }    
    
   \draw[line width=0.25mm](16,0) -- (16,0.25);
   \draw[line width=0.25mm](14,0) -- (14,0.25);
   \draw[line width=0.25mm](14,4) -- (14,3.75);
   \draw[line width=0.25mm](16,4) -- (16,3.75);
   \draw[line width=0.4mm](13.75,0.25) -- (16.25,0.25);
   \draw[line width=0.4mm](13.75,1.25) -- (16.25,1.25);
   \draw[line width=0.4mm](13.75,0.25) -- (13.75,1.25);
   \draw[line width=0.4mm](16.25,0.25) -- (16.25,1.25);
   \draw[line width=0.4mm](13.75,2.75) -- (16.25,2.75);
   \draw[line width=0.4mm](13.75,3.75) -- (16.25,3.75);
   \draw[line width=0.4mm](13.75,2.75) -- (13.75,3.75);
   \draw[line width=0.4mm](16.25,2.75) -- (16.25,3.75); 
   \draw[line width=0.25mm](14,1.25) -- (14,2.75);
   \draw[line width=0.25mm](15.25,1.25) -- (15.25,2.75);
   
   \draw[line width=0.25mm](17,0) -- (17,1.5);
   \draw[line width=0.25mm](17,4) -- (17,2.5);
   \draw[line width=0.25mm](16,1.25) -- (16,1.5);
   \draw[line width=0.25mm](16,2.75) -- (16,2.5);
   \draw[line width=0.25mm](16,1.5) to[bend left=90] (17,1.5);
   \draw[line width=0.25mm](16,2.5) to[bend right=90] (17,2.5);
   
   \draw[circle,fill](15,0.5)circle[radius=0mm]node[below]{ $\dots$};
   \draw[circle,fill](15,4.5)circle[radius=0mm]node[below]{ $\dots$};
   \draw[circle,fill](14.625,2.5)circle[radius=0mm]node[below]{ $\dots$};
   \draw[circle,fill](15,4.15)circle[radius=0mm]node[below]{ $f_{n-2}^{ij}$};
   \draw[circle,fill](15,1.65)circle[radius=0mm]node[below]{ $f_{n-2}^{ij}$};

    \draw[circle,fill](13.65,2.35)circle[radius=0mm]node[below]{ $i$};
    \draw[circle,fill](15.8,2.35)circle[radius=0mm]node[below]{ $j$};
    \draw[circle,fill](16.65,3.5)circle[radius=0mm]node[below]{ $i$};
    \draw[circle,fill](16.65,1)circle[radius=0mm]node[below]{ $i$};

    \draw[line width=0.25mm](17,4) to[bend left=90] (18,4);
    \draw[line width=0.25mm](17,0) to[bend right=90] (18,0);
    \draw[line width=0.25mm](18,4) to (18,0);
  \end{tikzpicture}%
}

\newcommand\MatchingBoxLOUCapFT{%
  \begin{tikzpicture}[scale = 0.75]
    \foreach \x in {0,2,3}{
       \draw[circle,fill] (\x,0)circle[radius=0.5mm]node[below]{};
       }
       
       \draw[circle,fill] (0,4)circle[radius=0.5mm]node[below]{};
   \draw[line width=0.25mm](0,0) -- (0,1);
   \draw[line width=0.25mm](2,0) -- (2,1);
   \draw[line width=0.25mm](0,4) -- (0,3);
   \draw[line width=0.25mm](2,4) -- (2,3);
   \draw[line width=0.25mm](3,0) -- (3,4);
   \draw[line width=0.4mm](-0.25,1) -- (2.25,1);
   \draw[line width=0.4mm](-0.25,3) -- (2.25,3);
   \draw[line width=0.4mm](-0.25,1) -- (-0.25,3);
   \draw[line width=0.4mm](2.25,1) -- (2.25,3);
   \draw[circle,fill](1,2.75)circle[radius=0mm]node[below]{  $f_{n-1}^{ij}$};
   \draw[circle,fill](1,1)circle[radius=0mm]node[below]{  $\dots$};
   \draw[circle,fill](1,4)circle[radius=0mm]node[below]{  $\dots$};
  
    \draw[circle,fill](-0.5,2.35)circle[radius=0mm]node[below]{  $i$};
    \draw[circle,fill](2.6,2.35)circle[radius=0mm]node[below]{  $i$};
    \draw[circle,fill](3.3,2.35)circle[radius=0mm]node[below]{  $j$};
   
    \draw[line width=0.25mm](2,4) to[bend left=90] (3,4);

  \end{tikzpicture}%
}

\newcommand\MatchingBoxLOUCapFL{%
  \begin{tikzpicture}[scale = 0.75]
    \foreach \x in {0,2,3}{
       \draw[circle,fill] (\x+7,0)circle[radius=0.5mm]node[below]{};
       }
       
       \draw[circle,fill] (7,4)circle[radius=0.5mm]node[below]{};
       \draw[circle,fill] (8.25,4)circle[radius=0.5mm]node[below]{};
   \draw[line width=0.25mm](8.25,4) -- (8.25,3.75);

   \draw[line width=0.25mm](7,0) -- (7,0.25);
   \draw[line width=0.25mm](9,0) -- (9,0.25);
   \draw[line width=0.25mm](7,4) -- (7,3.75);
   \draw[line width=0.4mm](6.75,0.25) -- (9.25,0.25);
   \draw[line width=0.4mm](6.75,1.25) -- (9.25,1.25);
   \draw[line width=0.4mm](6.75,0.25) -- (6.75,1.25);
   \draw[line width=0.4mm](9.25,0.25) -- (9.25,1.25);
   \draw[line width=0.4mm](6.75,2.75) -- (8.5,2.75);
   \draw[line width=0.4mm](6.75,3.75) -- (8.5,3.75);
   \draw[line width=0.4mm](6.75,2.75) -- (6.75,3.75);
   \draw[line width=0.4mm](8.5,2.75) -- (8.5,3.75); 
   \draw[line width=0.25mm](7,1.25) -- (7,2.75);
   \draw[line width=0.25mm](8.25,1.25) -- (8.25,2.75);
   
   \draw[line width=0.25mm](10,0) -- (10,1.5);
   \draw[line width=0.25mm](9,1.25) -- (9,1.5);
   \draw[line width=0.25mm](9,1.5) to[bend left=90] (10,1.5);

   \draw[circle,fill](8,0.5)circle[radius=0mm]node[below]{  $\dots$};
   \draw[circle,fill](7.625,4.5)circle[radius=0mm]node[below]{  $\dots$};
   \draw[circle,fill](7.625,2.5)circle[radius=0mm]node[below]{  $\dots$};
   \draw[circle,fill](7.625,4)circle[radius=0mm]node[below]{  $f_{n-2}^{ij}$};
   \draw[circle,fill](8,1.5)circle[radius=0mm]node[below]{  $f_{n-1}^{ij}$};
    \draw[circle,fill](6.65,2.35)circle[radius=0mm]node[below]{  $i$};
    \draw[circle,fill](9.5,3)circle[radius=0mm]node[below]{  $j$};
    \draw[circle,fill](9.65,1)circle[radius=0mm]node[below]{  $i$};

  \end{tikzpicture}%
}

\newcommand\MatchingBoxLOUA{%
  \begin{tikzpicture}[scale=0.75]
    \foreach \x in {1,3,4}{
       \draw[circle,fill] (\x,0)circle[radius=0.5mm]node[below]{};
       \draw[circle,fill] (\x,4)circle[radius=0.5mm]node[below]{};
       }
   \draw[line width=0.25mm](1,0) -- (1,1);
   \draw[line width=0.25mm](3,0) -- (3,1);
   \draw[line width=0.25mm](1,4) -- (1,3);
   \draw[line width=0.25mm](3,4) -- (3,3);
   \draw[line width=0.25mm](4,0) -- (4,4);
   \draw[line width=0.4mm](0.75,1) -- (3.25,1);
   \draw[line width=0.4mm](0.75,3) -- (3.25,3);
   \draw[line width=0.4mm](0.75,1) -- (0.75,3);
   \draw[line width=0.4mm](3.25,1) -- (3.25,3);
   \draw[circle,fill](2,2.75)circle[radius=0mm]node[below]{ \large $f_{n-1}^{ij}$};
   \draw[circle,fill](2,1)circle[radius=0mm]node[below]{  $\dots$};
   \draw[circle,fill](2,4)circle[radius=0mm]node[below]{  $\dots$};

    \draw[circle,fill](0.5,2.35)circle[radius=0mm]node[below]{  $i$};
    \draw[circle,fill](3.6,2.35)circle[radius=0mm]node[below]{  $i$};
    \draw[circle,fill](4.3,2.35)circle[radius=0mm]node[below]{  $j$};
  \end{tikzpicture}%
}

\newcommand\MatchingBoxLOUB{%
  \begin{tikzpicture}[scale=0.75]
    \foreach \x in {1,3,4}{
       \draw[circle,fill] (\x+7,0)circle[radius=0.5mm]node[below]{};
       \draw[circle,fill] (\x+7,4)circle[radius=0.5mm]node[below]{};
       }
   \draw[line width=0.25mm](8,0) -- (8,0.25);
   \draw[line width=0.25mm](10,0) -- (10,0.25);
   \draw[line width=0.25mm](8,4) -- (8,3.75);
   \draw[line width=0.25mm](10,4) -- (10,3.75);
   \draw[line width=0.4mm](7.75,0.25) -- (10.25,0.25);
   \draw[line width=0.4mm](7.75,1.25) -- (10.25,1.25);
   \draw[line width=0.4mm](7.75,0.25) -- (7.75,1.25);
   \draw[line width=0.4mm](10.25,0.25) -- (10.25,1.25);
   \draw[line width=0.4mm](7.75,2.75) -- (10.25,2.75);
   \draw[line width=0.4mm](7.75,3.75) -- (10.25,3.75);
   \draw[line width=0.4mm](7.75,2.75) -- (7.75,3.75);
   \draw[line width=0.4mm](10.25,2.75) -- (10.25,3.75); 
   \draw[line width=0.25mm](8,1.25) -- (8,2.75);
   \draw[line width=0.25mm](9.25,1.25) -- (9.25,2.75);
   
   \draw[line width=0.25mm](11,0) -- (11,1.5);
   \draw[line width=0.25mm](11,4) -- (11,2.5);
   \draw[line width=0.25mm](10,1.25) -- (10,1.5);
   \draw[line width=0.25mm](10,2.75) -- (10,2.5);
   \draw[line width=0.25mm](10,1.5) to[bend left=90] (11,1.5);
   \draw[line width=0.25mm](10,2.5) to[bend right=90] (11,2.5);
   
   \draw[circle,fill](9,0.5)circle[radius=0mm]node[below]{  $\dots$};
   \draw[circle,fill](9,4.5)circle[radius=0mm]node[below]{  $\dots$};
   \draw[circle,fill](8.625,2.5)circle[radius=0mm]node[below]{  $\dots$};
   \draw[circle,fill](9,4)circle[radius=0mm]node[below]{  $f_{n-1}^{ij}$};
   \draw[circle,fill](9,1.5)circle[radius=0mm]node[below]{  $f_{n-1}^{ij}$};
  
    \draw[circle,fill](7.65,2.35)circle[radius=0mm]node[below]{  $i$};
    \draw[circle,fill](9.8,2.35)circle[radius=0mm]node[below]{  $j$};
    \draw[circle,fill](10.65,3.5)circle[radius=0mm]node[below]{  $i$};
    \draw[circle,fill](10.65,1)circle[radius=0mm]node[below]{  $i$};

  \end{tikzpicture}%
}

\newcommand\MatchingBoxLOUC{%
  \begin{tikzpicture}[scale=0.75] 
\draw[line width=0.4mm](13.75,0.25) -- (15.25,0.25);
   \draw[line width=0.4mm](13.75,1) -- (15.25,1);
   \draw[line width=0.4mm](13.75,0.25) -- (13.75,1);
   \draw[line width=0.4mm](15.25,0.25) -- (15.25,1);

   \draw[line width=0.4mm](13.75,1.625) -- (15.25,1.625);
   \draw[line width=0.4mm](13.75,2.375) -- (15.25,2.375);
   \draw[line width=0.4mm](13.75,1.625) -- (13.75,2.375);
   \draw[line width=0.4mm](15.25,1.625) -- (15.25,2.375);
   
   \draw[line width=0.4mm](13.75,3) -- (15.25,3);
   \draw[line width=0.4mm](13.75,3.75) -- (15.25,3.75);
   \draw[line width=0.4mm](13.75,3) -- (13.75,3.75);
   \draw[line width=0.4mm](15.25,3) -- (15.25,3.75);

   \draw[line width=0.25mm](14,0) -- (14,0.25);
   \draw[line width=0.25mm](14,1) -- (14,1.625);
   \draw[line width=0.25mm](15,0) -- (15,0.25);
   \draw[line width=0.25mm](14.75,2.375) -- (14.75,3);
   \draw[line width=0.25mm](14,2.375) -- (14,3);
   \draw[line width=0.25mm](15,4) -- (15,3.75);
   \draw[line width=0.25mm](14,4) -- (14,3.75);
   \draw[line width=0.25mm](14.75,1) -- (14.75,1.625);
   \draw[line width=0.25mm](15,2.9) to[bend right=60] (15.5,2.9);
   \draw[line width=0.25mm](15,2.475) to[bend left=60] (15.5,2.475);
   \draw[line width=0.25mm](15,1.525) to[bend right=60] (15.5,1.525);
   \draw[line width=0.25mm](15,1.1) to[bend left=60] (15.5,1.1);
   \draw[line width=0.25mm](15.5,2.9) -- (15.5,4);
   \draw[line width=0.25mm](15.5,1.1) -- (15.5,0);
   \draw[line width=0.25mm](15.5,1.525) -- (15.5,2.475);
   \draw[line width=0.25mm](15,1) -- (15,1.1);
   \draw[line width=0.25mm](15,1.525) -- (15,1.625);
   \draw[line width=0.25mm](15,2.375) -- (15,2.475);
   \draw[line width=0.25mm](15,2.9) -- (15,3);

   \draw[circle,fill](14.5,4.2)circle[radius=0mm]node[below]{  $f_{n-1}^{ij}$};
   \draw[circle,fill](14.5,1.5)circle[radius=0mm]node[below]{  $f_{n-1}^{ij}$};
   \draw[circle,fill](14.5,2.85)circle[radius=0mm]node[below]{  $f_{n-1}^{ij}$};   
   \draw[circle,fill](13.6,2.85)circle[radius=0mm]node[below]{  $i$};
\draw[circle,fill](15.3,4.45)circle[radius=0mm]node[below]{  $i$};
\draw[circle,fill](15.3,0.45)circle[radius=0mm]node[below]{  $i$};
\draw[circle,fill](15.38,2.35)circle[radius=0mm]node[below]{  $i$};
\draw[circle,fill](15.75,2.35)circle[radius=0mm]node[below]{  $j$};

\draw[circle,fill](14.5,0.55)circle[radius=0mm]node[below]{  $\dots$};
   \draw[circle,fill](14.4,3.125)circle[radius=0mm]node[below]{  $\dots$};
   \draw[circle,fill](14.4,1.85)circle[radius=0mm]node[below]{  $\dots$};
   \draw[circle,fill](14.5,4.4)circle[radius=0mm]node[below]{  $\dots$};
   
   \draw[circle,fill] (14,4)circle[radius=0.5mm]node[below]{};
   \draw[circle,fill] (15,4)circle[radius=0.5mm]node[below]{};
   \draw[circle,fill] (14,0)circle[radius=0.5mm]node[below]{};
   \draw[circle,fill] (15,0)circle[radius=0.5mm]node[below]{};
   \draw[circle,fill] (15.5,0)circle[radius=0.5mm]node[below]{};
   \draw[circle,fill] (15.5,4)circle[radius=0.5mm]node[below]{};
   
  \end{tikzpicture}%
}

\newcommand\MatchingOnes{%
  \begin{tikzpicture}[thick,scale=0.6, every node/.style={scale=0.6}]
    \foreach \x in {1,2,4,5}{
       \draw[circle,fill] (\x,0)circle[radius=0.5mm]node[below]{};
       \draw[circle,fill] (\x,2)circle[radius=0.5mm]node[below]{};
       \draw[line width=0.25mm](\x,2) -- (\x,0);
       }
    \draw[line width=0.25mm](5,0) -- (5,0.5);
    \draw[line width=0.25mm](4,0) -- (4,0.5);
	\draw[circle,fill] (5,0)circle[radius=0.5mm]node[below]{};
     \draw[circle,fill] (4,0)circle[radius=0.5mm]node[below]{};
    
	\draw[circle,fill] (0.5,1.5)circle[radius=0mm]node[below]{\Large $i_{1}$};
	\draw[circle,fill] (1.5,1.5)circle[radius=0mm]node[below]{\Large $i_{2}$};
	\draw[circle,fill] (3,1.5)circle[radius=0mm]node[below]{\Large $\dots$};
	\draw[circle,fill]  (4.5,1.5)circle[radius=0mm]node[below]{\Large $i_{n}$};
	\draw[circle,fill] (5.6,1.7)circle[radius=0mm]node[below]{\Large $i_{n+1}$};

  \end{tikzpicture}%
}

\newcommand\GeneratoreEven{\begin{tikzpicture}[scale = 0.75]
    \foreach \x in {1,3,4}{
       \draw[circle,fill] (\x+7,0)circle[radius=0.5mm]node[below]{};
       \draw[circle,fill] (\x+7,4)circle[radius=0.5mm]node[below]{};
       }
   \draw[line width=0.25mm](8,0) -- (8,0.25);
   \draw[line width=0.25mm](10,0) -- (10,0.25);
   \draw[line width=0.25mm](8,4) -- (8,3.75);
   \draw[line width=0.25mm](10,4) -- (10,3.75);
   \draw[line width=0.4mm](7.75,0.25) -- (10.25,0.25);
   \draw[line width=0.4mm](7.75,1.25) -- (10.25,1.25);
   \draw[line width=0.4mm](7.75,0.25) -- (7.75,1.25);
   \draw[line width=0.4mm](10.25,0.25) -- (10.25,1.25);
   \draw[line width=0.4mm](7.75,2.75) -- (10.25,2.75);
   \draw[line width=0.4mm](7.75,3.75) -- (10.25,3.75);
   \draw[line width=0.4mm](7.75,2.75) -- (7.75,3.75);
   \draw[line width=0.4mm](10.25,2.75) -- (10.25,3.75); 
   \draw[line width=0.25mm](8,1.25) -- (8,2.75);
   \draw[line width=0.25mm](9.25,1.25) -- (9.25,2.75);
   
   \draw[line width=0.25mm](11,0) -- (11,1.5);
   \draw[line width=0.25mm](11,4) -- (11,2.5);
   \draw[line width=0.25mm](10,1.25) -- (10,1.5);
   \draw[line width=0.25mm](10,2.75) -- (10,2.5);
   \draw[line width=0.25mm](10,1.5) to[bend left=90] (11,1.5);
   \draw[line width=0.25mm](10,2.5) to[bend right=90] (11,2.5);
   
   \draw[circle,fill](9,0.5)circle[radius=0mm]node[below]{  $\dots$};
   \draw[circle,fill](9,4.5)circle[radius=0mm]node[below]{  $\dots$};
   \draw[circle,fill](8.625,2.5)circle[radius=0mm]node[below]{  $\dots$};
   \draw[circle,fill](9,3.9)circle[radius=0mm]node[below]{  $f_{n}^{ij}$};
   \draw[circle,fill](9,1.4)circle[radius=0mm]node[below]{  $f_{n}^{ij}$};
    \draw[circle,fill](7.65,2.35)circle[radius=0mm]node[below]{  $i$};
    \draw[circle,fill](9.8,2.35)circle[radius=0mm]node[below]{  $j$};
    \draw[circle,fill](10.65,3.5)circle[radius=0mm]node[below]{  $i$};
    \draw[circle,fill](10.65,1)circle[radius=0mm]node[below]{  $i$};  
  \end{tikzpicture}}

\newcommand\ZeroSquareEven{
\begin{tikzpicture}[scale=0.85]
    \foreach \y in {0}{
        \foreach \x in {\y+1,\y+2.5,\y+3,\y+4}{      
       \draw[circle,fill] (\x,0)circle[radius=0.5mm]node[below]{};
       \draw[circle,fill] (\x,4)circle[radius=0.5mm]node[below]{};}   
        \foreach \x in {\y+1,\y+2.5, \y+3}{
        \draw[line width=0.25mm](\x,0) -- (\x,0.25);
        \draw[line width=0.25mm](\x,4) -- (\x,3.75);}

        \draw[line width=0.25mm](\y+4,0) -- (\y+4,1);   
        \draw[line width=0.25mm](\y+4,4) -- (\y+4,3);

    \draw[line width=0.4mm](\y+0.5,0.25) -- (\y+3.5,0.25);
   \draw[line width=0.4mm](\y+0.5,1) -- (\y+3.5,1);
   \draw[line width=0.4mm](\y+0.5,0.25) -- (\y+0.5,1);
   \draw[line width=0.4mm](\y+3.5,0.25) -- (\y+3.5,1);
   \draw[line width=0.4mm](\y+0.5,3) -- (\y+3.5,3);
   \draw[line width=0.4mm](\y+0.5,3.75) -- (\y+3.5,3.75);
   \draw[line width=0.4mm](\y+0.5,3) -- (\y+0.5,3.75);
   \draw[line width=0.4mm](\y+3.5,3) -- (\y+3.5,3.75);
   \draw[line width=0.25mm](\y+3,1) to[bend left=60] (\y+4,1);
   \draw[line width=0.25mm](\y+3,3) to[bend right=60] (\y+4,3);
   \draw[circle,fill](\y+2,3.95)circle[radius=0mm]node[below]{  $f_{n}^{ij}$};
   \draw[circle,fill](\y+2,1.2)circle[radius=0mm]node[below]{  $f_{n}^{ij}$};
   \draw[circle,fill](\y+0.3,2.85)circle[radius=0mm]node[below]{  $i$};
\draw[circle,fill](\y+3.6,4.45)circle[radius=0mm]node[below]{  $i$};
\draw[circle,fill](\y+3.6,0.45)circle[radius=0mm]node[below]{  $i$};
\draw[circle,fill](\y+4.5,2.35)circle[radius=0mm]node[below]{  $j$};
\draw[circle,fill](\y+1.75,0.55)circle[radius=0mm]node[below]{  $\dots$};
   \draw[circle,fill](\y+1.75,3.2)circle[radius=0mm]node[below]{  $\dots$};
   \draw[circle,fill](\y+1.75,1.85)circle[radius=0mm]node[below]{  $\dots$};
   \draw[circle,fill](\y+1.75,4.4)circle[radius=0mm]node[below]{  $\dots$};

    }

   \draw[line width=0.4mm](0.5,1.625) -- (3.5,1.625);
   \draw[line width=0.4mm](0.5,2.375) -- (3.5,2.375);
   \draw[line width=0.4mm](0.5,1.625) -- (0.5,2.375);
   \draw[line width=0.4mm](3.5,1.625) -- (3.5,2.375);   
   \draw[line width=0.25mm](3,2.375) to[bend left=60] (4,2.375);
   \draw[line width=0.25mm](3,1.625) to[bend right=60] (4,1.625); 
    \draw[line width=0.25mm](4,1.625) -- (4,2.375);
   \draw[circle,fill](2,2.55)circle[radius=0mm]node[below]{  $f_{n}^{ij}$};      
\draw[circle,fill](3.68,2.35)circle[radius=0mm]node[below]{  $i$};
\draw[line width=0.25mm](1,1) -- (1,1.625);
\draw[line width=0.25mm](1,2.375) -- (1,3);
\draw[line width=0.25mm](2.5,1) -- (2.5,1.625);
\draw[line width=0.25mm](2.5,2.375) -- (2.5,3);
\end{tikzpicture}}

\newcommand\ZeroSquareEvenTwo{

\begin{tikzpicture}
\foreach \y in {10}{
        \foreach \x in {\y+1,\y+2.5,\y+3,\y+4}{      
       \draw[circle,fill] (\x,0-5)circle[radius=0.5mm]node[below]{};
       \draw[circle,fill] (\x,4-5)circle[radius=0.5mm]node[below]{};}   
        \foreach \x in {\y+1,\y+2.5, \y+3}{
        \draw[line width=0.25mm](\x,0-5) -- (\x,0.25-5);
        \draw[line width=0.25mm](\x,4-5) -- (\x,3.75-5);}

        \draw[line width=0.25mm](\y+4,0-5) -- (\y+4,1-5);   
        \draw[line width=0.25mm](\y+4,4-5) -- (\y+4,3-5);

    \draw[line width=0.4mm](\y+0.5,0.25-5) -- (\y+3.5,0.25-5);
   \draw[line width=0.4mm](\y+0.5,1-5) -- (\y+3.5,1-5);
   \draw[line width=0.4mm](\y+0.5,0.25-5) -- (\y+0.5,1-5);
   \draw[line width=0.4mm](\y+3.5,0.25-5) -- (\y+3.5,1-5);
   \draw[line width=0.4mm](\y+0.5,3-5) -- (\y+3.5,3-5);
   \draw[line width=0.4mm](\y+0.5,3.75-5) -- (\y+3.5,3.75-5);
   \draw[line width=0.4mm](\y+0.5,3-5) -- (\y+0.5,3.75-5);
   \draw[line width=0.4mm](\y+3.5,3-5) -- (\y+3.5,3.75-5);
   \draw[line width=0.25mm](\y+3,1-5) to[bend left=60] (\y+4,1-5);
   \draw[line width=0.25mm](\y+3,3-5) to[bend right=60] (\y+4,3-5);
   \draw[circle,fill](\y+2,3.95-5)circle[radius=0mm]node[below]{  $f_{n}^{ij}$};
   \draw[circle,fill](\y+2,1.2-5)circle[radius=0mm]node[below]{  $f_{n}^{ij}$};
   \draw[circle,fill](\y+0.3,2.85-5)circle[radius=0mm]node[below]{  $i$};
\draw[circle,fill](\y+3.6,4.45-5)circle[radius=0mm]node[below]{  $i$};
\draw[circle,fill](\y+3.6,0.45-5)circle[radius=0mm]node[below]{  $i$};
\draw[circle,fill](\y+4.5,2.35-5)circle[radius=0mm]node[below]{  $j$};
\draw[circle,fill](\y+1.75,0.45-5)circle[radius=0mm]node[below]{  $\dots$};
   \draw[circle,fill](\y+1.75,3.15-5)circle[radius=0mm]node[below]{  $\dots$};
   \draw[circle,fill](\y+1.75,1.65-5)circle[radius=0mm]node[below]{  $\dots$};
   \draw[circle,fill](\y+1.75,4.3-5)circle[radius=0mm]node[below]{  $\dots$};
    }
\draw[line width=0.4mm](10+0.5,1.625-5) -- (10+2.75,1.625-5);
   \draw[line width=0.4mm](10+0.5,2.375-5) -- (10+2.75,2.375-5);
   \draw[line width=0.4mm](10+0.5,1.625-5) -- (10+0.5,2.375-5);
   \draw[line width=0.4mm](10+2.75,1.625-5) -- (10+2.75,2.375-5);
   \draw[circle,fill](11.6,2.7-5)circle[radius=0mm]node[below]{  $f_{n-1}^{ij}$};      
\draw[line width=0.25mm](10+1,1-5) -- (10+1,1.625-5);
\draw[line width=0.25mm](10+1,2.375-5) -- (10+1,3-5);
\draw[line width=0.25mm](10+2.5,1-5) -- (10+2.5,1.625-5);
\draw[line width=0.25mm](10+2.5,2.375-5) -- (10+2.5,3-5);
\end{tikzpicture}%
}

\newcommand\TLMultiply{
\begin{tikzpicture}[scale = 0.75]
    \foreach \x in {1,2,3,5,6,7,13,14,15}{
       \draw[circle,fill] (\x,0)circle[radius=0.5mm]node[below]{};
       \draw[circle,fill] (\x,2)circle[radius=0.5mm]node[below]{};
       }
   \draw[line width=0.25mm](1,0) -- (3,2);
   \draw[line width=0.25mm](1,2) to[bend right=90] (2,2);
   \draw[line width=0.25mm](2,0) to[bend left=90] (3,0);
   \draw[circle,fill](1,1.5)circle[radius=0mm]node[below]{  $a$};
   \draw[circle,fill](1.5,2.5)circle[radius=0mm]node[below]{  $b$};
   \draw[circle,fill](2.5,0.5)circle[radius=0mm]node[below]{  $c$};
   \draw[circle,fill](3,1.5)circle[radius=0mm]node[below]{  $d$};
   
   \draw[circle,fill](4,1.5)circle[radius=0mm]node[below]{  $\circ$};

   \draw[line width=0.25mm](5,2) -- (7,0);
   \draw[line width=0.25mm](5,0) to[bend left=90] (6,0);
   \draw[line width=0.25mm](6,2) to[bend right=90] (7,2);
   \draw[circle,fill](5,1.5)circle[radius=0mm]node[below]{  $e$};
   \draw[circle,fill](5.5,0.5)circle[radius=0mm]node[below]{  $f$};
   \draw[circle,fill](6.5,2.5)circle[radius=0mm]node[below]{  $g$};
   \draw[circle,fill](7,1.5)circle[radius=0mm]node[below]{  $h$};
   
   \draw[circle,fill](8,1.5)circle[radius=0mm]node[below]{  $=$};
   
    \foreach \x in {9,10,11}{
       \draw[circle,fill] (\x,-1)circle[radius=0.5mm]node[below]{};
       \draw[circle,fill] (\x,1)circle[radius=0.5mm]node[below]{};
       \draw[circle,fill] (\x,3)circle[radius=0.5mm]node[below]{};
       }
        \draw[line width=0.25mm](9,1) -- (11,3);
        \draw[line width=0.25mm](9,3) to[bend right=90] (10,3);
        \draw[line width=0.25mm](10,1) to[bend left=90] (11,1);
        \draw[line width=0.25mm](9,1) -- (11,-1);
        \draw[line width=0.25mm](9,-1) to[bend left=90] (10,-1);
        \draw[line width=0.25mm](10,1) to[bend right=90] (11,1);
        \draw[circle,fill](9,2.5)circle[radius=0mm]node[below]{  $ae$};
        \draw[circle,fill](9.5,3.5)circle[radius=0mm]node[below]{  $b$};
        \draw[circle,fill](10.5,1.5)circle[radius=0mm]node[below]{  $cg$};
        \draw[circle,fill](11,2.5)circle[radius=0mm]node[below]{  $dh$};
        \draw[circle,fill](9.5,-0.5)circle[radius=0mm]node[below]{  $f$};
        
   \draw[circle,fill](12,1.5)circle[radius=0mm]node[below]{  $=$};
   \draw[line width=0.25mm](15,0) -- (15,2);
   \draw[line width=0.25mm](13,2) to[bend right=90] (14,2);
   \draw[line width=0.25mm](13,0) to[bend left=90] (14,0);
   \draw[circle,fill](13.75,1.5)circle[radius=0mm]node[below]{  $ae$};
   \draw[circle,fill](13.5,2.5)circle[radius=0mm]node[below]{  $b$};
   \draw[circle,fill](13.5,0.5)circle[radius=0mm]node[below]{  $f$};
   \draw[circle,fill](16,2.1)circle[radius=0mm]node[below]{  $\omega(cg)dh$};
  \end{tikzpicture}}

\newcommand\TLMultilinear{\begin{tikzpicture}[scale = 0.65]
    \foreach \x in {1,2,3,5,6,7,10,11,12}{
       \draw[circle,fill] (\x,0)circle[radius=0.5mm]node[below]{};
       \draw[circle,fill] (\x,2)circle[radius=0.5mm]node[below]{};
       }
   \draw[line width=0.25mm](1,0) -- (3,2);
   \draw[line width=0.25mm](1,2) to[bend right=90] (2,2);
   \draw[line width=0.25mm](2,0) to[bend left=90] (3,0);
   \draw[circle,fill](1,2.3)circle[radius=0mm]node[below]{  $a+\lambda b$};
   \draw[circle,fill](1.5,2.5)circle[radius=0mm]node[below]{  $c$};
   \draw[circle,fill](2.5,0.5)circle[radius=0mm]node[below]{  $d$};
   \draw[circle,fill](3,1.5)circle[radius=0mm]node[below]{  $e$};
   
   \draw[circle,fill](4,1.5)circle[radius=0mm]node[below]{  $=$};
    \draw[line width=0.25mm](5,0) -- (7,2);
   \draw[line width=0.25mm](5,2) to[bend right=90] (6,2);
   \draw[line width=0.25mm](6,0) to[bend left=90] (7,0);
   \draw[circle,fill](5,1.5)circle[radius=0mm]node[below]{  $a$};
   \draw[circle,fill](5.5,2.5)circle[radius=0mm]node[below]{  $c$};
   \draw[circle,fill](6.5,0.5)circle[radius=0mm]node[below]{  $d$};
   \draw[circle,fill](7,1.5)circle[radius=0mm]node[below]{  $e$};
    
   \draw[circle,fill](8.4,1.7)circle[radius=0mm]node[below]{  $+\,\,\,\,\lambda$};
   \draw[line width=0.25mm](13-3,0) -- (15-3,2);
   \draw[line width=0.25mm](13-3,2) to[bend right=90] (14-3,2);
   \draw[line width=0.25mm](14-3,0) to[bend left=90] (15-3,0);
   \draw[circle,fill](13-3,1.5)circle[radius=0mm]node[below]{  $b$};
   \draw[circle,fill](13.5-3,2.5)circle[radius=0mm]node[below]{  $c$};
   \draw[circle,fill](14.5-3,0.5)circle[radius=0mm]node[below]{  $d$};
   \draw[circle,fill](15-3,1.5)circle[radius=0mm]node[below]{  $e$};
   
  \end{tikzpicture}}

\newcommand\TLGeneratorh{\begin{tikzpicture}[scale = 1.2]
       \draw[circle,fill] (1,0)circle[radius=0.5mm]node[below]{};
       \draw[circle,fill] (1,2)circle[radius=0.5mm]node[above]{};
       \draw[line width=0.25mm](1,2) -- (1,0);
       \draw[circle,fill] (2,0)circle[radius=0.5mm]node[below]{};
       \draw[circle,fill] (2,2)circle[radius=0.5mm]node[above]{};
       \draw[line width=0.25mm](2,2) -- (2,0);
       \draw[circle,fill] (4,0)circle[radius=0.5mm]node[below]{};
       \draw[circle,fill] (4,2)circle[radius=0.5mm]node[above]{};
       \draw[line width=0.25mm](4,2) -- (4,0);
       \draw[circle,fill] (7,0)circle[radius=0.5mm]node[below]{};
       \draw[circle,fill] (7,2)circle[radius=0.5mm]node[above]{};
       \draw[line width=0.25mm](7,2) -- (7,0);
       \draw[circle,fill] (9,0)circle[radius=0.5mm]node[below]{};
       \draw[circle,fill] (9,2)circle[radius=0.5mm]node[above]{};
       \draw[line width=0.25mm](9,2) -- (9,0);
       \draw[circle,fill] (10,0)circle[radius=0.5mm]node[below]{};
       \draw[circle,fill] (10,2)circle[radius=0.5mm]node[above]{};
       \draw[line width=0.25mm](10,2) -- (10,0);
       \draw[circle,fill] (5,0)circle[radius=0.5mm]node[below]{};
       \draw[circle,fill] (5,2)circle[radius=0.5mm]node[above]{};
       \draw[circle,fill] (6,0)circle[radius=0.5mm]node[below]{};
       \draw[circle,fill] (6,2)circle[radius=0.5mm]node[above]{};
       \draw[line width=0.25mm](5,2) to[bend right=90] (6,2);
       \draw[line width=0.25mm](5,0) to[bend left=90] (6,0);

	\draw[circle,fill] (0.5,1.5)circle[radius=0mm]node[below]{ $i_{1}$};
	\draw[circle,fill] (1.5,1.5)circle[radius=0mm]node[below]{ $i_{2}$};
	\draw[circle,fill] (3,1.5)circle[radius=0mm]node[below]{ $\dots$};
    \draw[circle,fill] (8,1.5)circle[radius=0mm]node[below]{ $\dots$};
	\draw[circle,fill] (4.5,1.5)circle[radius=0mm]node[below]{ $i_{k}$};
	\draw[circle,fill] (5.5,2.5)circle[radius=0mm]node[below]{$i_{k+1}$};
    \draw[circle,fill] (5.5,0.5)circle[radius=0mm]node[below]{$i_{k+1}$};
    \draw[circle,fill] (6.5,1.64)circle[radius=0mm]node[below]{$i_{k+2}$};
    \draw[circle,fill] (9.5,1.5)circle[radius=0mm]node[below]{$i_{n}$};
    \draw[circle,fill] (10.5,1.64)circle[radius=0mm]node[below]{$i_{n+1},$};

    \draw[circle,fill] (1,0)circle[radius=0mm]node[below]{$1$};
    \draw[circle,fill] (1,2)circle[radius=0mm]node[above]{$1$};
    \draw[circle,fill] (2,0)circle[radius=0mm]node[below]{$2$};
    \draw[circle,fill] (2,2)circle[radius=0mm]node[above]{$2$};
    \draw[circle,fill] (4,0.2)circle[radius=0mm]node[below]{$k-1$};
    \draw[circle,fill] (4,1.8)circle[radius=0mm]node[above]{$k-1$};
    \draw[circle,fill] (5,0)circle[radius=0mm]node[below]{$k$};
    \draw[circle,fill] (5,2)circle[radius=0mm]node[above]{$k$};
    \draw[circle,fill] (6,0.2)circle[radius=0mm]node[below]{$k+1$};
    \draw[circle,fill] (6,1.8)circle[radius=0mm]node[above]{$k+1$};
    \draw[circle,fill] (7,0.2)circle[radius=0mm]node[below]{$k+2$};
    \draw[circle,fill] (7,1.8)circle[radius=0mm]node[above]{$k+2$};
    \draw[circle,fill] (9,0.2)circle[radius=0mm]node[below]{$n-1$};
    \draw[circle,fill] (9,1.8)circle[radius=0mm]node[above]{$n-1$};
    \draw[circle,fill] (10,0)circle[radius=0mm]node[below]{$n$};
    \draw[circle,fill] (10,2)circle[radius=0mm]node[above]{$n$};
  \end{tikzpicture}}
  
\newcommand\TLNonIdentityFactorA{\begin{tikzpicture}[scale = 0.7]
    \fill[blue!20](0,0) rectangle (5,2);
    \fill[red!20] (1,0) -- (3,2) -- (4,2) -- (2,0);
    \fill[red!20] (1.5,2) ellipse (0.5 and 0.3);
    \fill[red!20] (3.5,0) ellipse (0.5 and 0.3);
    \fill[white] (1,-0.3) rectangle (5,0);
    \fill[white] (1, 2.3) rectangle (5,2);
    \foreach \x in {1,2,3,4}{
       \draw[circle,fill] (\x,0)circle[radius=0.5mm]node[below]{};
       }
    \foreach \x in {1,2,3,4}{
       \draw[circle,fill] (\x,2)circle[radius=0.5mm]node[below]{};
       }
    \foreach \x in {0,2,4}{
       \draw[circle,fill](\x + 0.5,2.5)circle[radius=0mm]node[below]{  $1$};
       }
    \foreach \x in {1,3}{
       \draw[circle,fill](\x + 0.5,2.5)circle[radius=0mm]node[below]{  $2$};
       }
    \foreach \x in {0,2,4}{
       \draw[circle,fill](\x + 0.5,0.5)circle[radius=0mm]node[below]{  $1$};
       }
    \foreach \x in {1,3}{
       \draw[circle,fill](\x + 0.5,0.5)circle[radius=0mm]node[below]{  $2$};
       }  
   \draw[line width=0.25mm](1,0) -- (3,2);
   \draw[line width=0.25mm](2,0) -- (4,2);
   \draw[line width=0.25mm, dashed](0,1) -- (5,1);
   \draw[line width=0.25mm](1,2) to[bend right=90] (2,2);
   \draw[line width=0.25mm](3,0) to[bend left=90] (4,0);

  \end{tikzpicture}}

\newcommand\TLNonIdentityFactorB{\begin{tikzpicture}[scale = 0.7]
    
    \fill[blue!20](6.5,0) rectangle (10.5,2);
    \fill[red!20] (7,0) -- (8,2) -- (9,2) -- (8,0);
    \fill[red!20] (9.5,0) ellipse (0.5 and 0.3);
    \fill[white] (9,-0.3) rectangle (10,0);
    \foreach \x in {7,8,9,10}{
       \draw[circle,fill] (\x,0)circle[radius=0.5mm]node[below]{};
       }
    \foreach \x in {8,9}{
       \draw[circle,fill] (\x,2)circle[radius=0.5mm]node[below]{};
       }
    \foreach \x in {7,9}{
       \draw[circle,fill](\x + 0.5,2.5)circle[radius=0mm]node[below]{  $1$};
       }
    \foreach \x in {8}{
       \draw[circle,fill](\x + 0.5,2.5)circle[radius=0mm]node[below]{  $2$};
       }
    \foreach \x in {6,8,10}{
       \draw[circle,fill](\x + 0.5,0.5)circle[radius=0mm]node[below]{  $1$};
       }
    \foreach \x in {7,9}{
       \draw[circle,fill](\x + 0.5,0.5)circle[radius=0mm]node[below]{  $2$};
       }  
    
   \draw[line width=0.25mm](12-5,0) -- (13-5,2);
   \draw[line width=0.25mm](13-5,0) -- (14-5,2);
   \draw[line width=0.25mm](14-5,0) to[bend left=90] (15-5,0);  
 
  \end{tikzpicture}}

\newcommand\TLNonIdentityFactorC{\begin{tikzpicture}[scale = 0.7]
    \fill[blue!20](11.5,0) rectangle (15.5,2);
    \fill[red!20] (13,0) -- (14,2) -- (15,2) -- (14,0);
    \fill[red!20] (12.5,2) ellipse (0.5 and 0.3);
    \fill[white] (11,-0.3) rectangle (11,0);
    \fill[white] (11, 2.3) rectangle (14,2);
    \foreach \x in {13,14}{
       \draw[circle,fill] (\x,0)circle[radius=0.5mm]node[below]{};
       }
    \foreach \x in {12,13,14,15}{
       \draw[circle,fill] (\x,2)circle[radius=0.5mm]node[below]{};
       }
    \foreach \x in {11,13,15}{
       \draw[circle,fill](\x + 0.5,2.5)circle[radius=0mm]node[below]{  $1$};
       }
    \foreach \x in {12,14}{
       \draw[circle,fill](\x + 0.5,2.5)circle[radius=0mm]node[below]{  $2$};
       }
    \foreach \x in {12,14}{
       \draw[circle,fill](\x + 0.5,0.5)circle[radius=0mm]node[below]{  $1$};
       }
    \foreach \x in {13}{
       \draw[circle,fill](\x + 0.5,0.5)circle[radius=0mm]node[below]{  $2$};
       }  
   
   \draw[line width=0.25mm](8+5,0) -- (9+5,2);
   \draw[line width=0.25mm](9+5,0) -- (10+5,2);
   \draw[line width=0.25mm](7+5,2) to[bend right=90] (8+5,2);

  \end{tikzpicture}}

\newcommand\TLMonoidalProductA{\begin{tikzpicture}[scale = 0.68]
    \fill[red!20](0,0) rectangle (4,2);
    \fill[blue!20] (0,0) -- (0,2) -- (3,2) -- (1,0);
    \fill[red!20] (1.5,2) ellipse (0.5 and 0.3);
    \fill[blue!20] (2.5,0) ellipse (0.5 and 0.3);
    \fill[white] (1,-0.3) rectangle (4,0);
    \fill[white] (1, 2.3) rectangle (4,2);
    \foreach \x in {1,2,3}{
       \draw[circle,fill] (\x,0)circle[radius=0.5mm]node[below]{};
       \draw[circle,fill] (\x,2)circle[radius=0.5mm]node[below]{};
       }
    \foreach \x in {0,2}{
       \draw[circle,fill](\x + 0.5,2.5)circle[radius=0mm]node[below]{  $1$};
       \draw[circle,fill](\x + 0.5,0.5)circle[radius=0mm]node[below]{  $1$};
       }
    \foreach \x in {1,3}{
       \draw[circle,fill](\x + 0.5,2.5)circle[radius=0mm]node[below]{  $2$};
       \draw[circle,fill](\x + 0.5,0.5)circle[radius=0mm]node[below]{  $2$};
       }
    \foreach \x in {}{
       \draw[circle,fill](\x + 0.5,2.5)circle[radius=0mm]node[below]{  $3$};
       \draw[circle,fill](\x + 0.5,0.5)circle[radius=0mm]node[below]{  $3$};
       } 
   \draw[line width=0.25mm](1,0) -- (3,2);
   \draw[line width=0.25mm](1,2) to[bend right=90] (2,2);
   \draw[line width=0.25mm](2,0) to[bend left=90] (3,0);
  \end{tikzpicture}}

\newcommand\TLMonoidalProductB{\begin{tikzpicture}[scale = 0.68]
    \fill[red!20](5,0) rectangle (9,2);
    \fill[red!20] (8,0) -- (8,2) -- (9,2) -- (9,0);
    \fill[green!20] (6.5,0) ellipse (0.5 and 0.3);
    \fill[green!20] (6.5,2) ellipse (0.5 and 0.3);
    \fill[white] (5,-0.3) rectangle (9,0);
    \fill[white] (5, 2.3) rectangle (9,2);
    \foreach \x in {6,7,8}{
       \draw[circle,fill] (\x,0)circle[radius=0.5mm]node[below]{};
       \draw[circle,fill] (\x,2)circle[radius=0.5mm]node[below]{};
       }
    \foreach \x in {}{
       \draw[circle,fill](\x + 0.5,2.5)circle[radius=0mm]node[below]{  $1$};
       \draw[circle,fill](\x + 0.5,0.5)circle[radius=0mm]node[below]{  $1$};
       }
    \foreach \x in {5,7,8}{
       \draw[circle,fill](\x + 0.5,2.5)circle[radius=0mm]node[below]{  $2$};
       \draw[circle,fill](\x + 0.5,0.5)circle[radius=0mm]node[below]{  $2$};
       }
    \foreach \x in {6}{
       \draw[circle,fill](\x + 0.5,2.5)circle[radius=0mm]node[below]{  $3$};
       \draw[circle,fill](\x + 0.5,0.5)circle[radius=0mm]node[below]{  $3$};
       } 
   \draw[line width=0.25mm](8,0) -- (8,2);
   \draw[line width=0.25mm](6,2) to[bend right=90] (7,2);
   \draw[line width=0.25mm](6,0) to[bend left=90] (7,0);  
  \end{tikzpicture}}

\newcommand\TLMonoidalProductC{\begin{tikzpicture}[scale = 0.68]
    \fill[red!20](10,0) rectangle (17,2);
    \fill[blue!20] (10,0) -- (10,2) -- (13,2) -- (11,0);
    \fill[red!20] (16,0) -- (16,2) -- (17,2) -- (17,0);
    \fill[red!20] (11.5,2) ellipse (0.5 and 0.3);
    \fill[blue!20] (12.5,0) ellipse (0.5 and 0.3);
    \fill[green!20] (14.5,0) ellipse (0.5 and 0.3);
    \fill[green!20] (14.5,2) ellipse (0.5 and 0.3);
    \fill[white] (10,-0.3) rectangle (16,0);
    \fill[white] (10, 2.3) rectangle (16,2);
    \foreach \x in {11,12,13,14,15,16}{
       \draw[circle,fill] (\x,0)circle[radius=0.5mm]node[below]{};
       \draw[circle,fill] (\x,2)circle[radius=0.5mm]node[below]{};
       }
    \foreach \x in {10,12}{
       \draw[circle,fill](\x + 0.5,2.5)circle[radius=0mm]node[below]{  $1$};
       \draw[circle,fill](\x + 0.5,0.5)circle[radius=0mm]node[below]{  $1$};
       }
    \foreach \x in {11,13,15,16}{
       \draw[circle,fill](\x + 0.5,2.5)circle[radius=0mm]node[below]{  $2$};
       \draw[circle,fill](\x + 0.5,0.5)circle[radius=0mm]node[below]{  $2$};
       }
    \foreach \x in {14}{
       \draw[circle,fill](\x + 0.5,2.5)circle[radius=0mm]node[below]{  $3$};
       \draw[circle,fill](\x + 0.5,0.5)circle[radius=0mm]node[below]{  $3$};
       }
   \draw[line width=0.25mm](11,0) -- (13,2);
   \draw[line width=0.25mm](11,2) to[bend right=90] (12,2);
   \draw[line width=0.25mm](12,0) to[bend left=90] (13,0);  
   \draw[line width=0.25mm](16,0) -- (16,2);
   \draw[line width=0.25mm](14,2) to[bend right=90] (15,2);
   \draw[line width=0.25mm](14,0) to[bend left=90] (15,0);      
  \end{tikzpicture}}

\newcommand\BkBasis{%
  \begin{tikzpicture}[scale = 0.55]
    \foreach \x in {1,4,6,9}{
       \draw[circle,fill] (\x,0+1.5)circle[radius=0.5mm]node[below]{};
       \draw[circle,fill] (\x,9)circle[radius=0.5mm]node[below]{};
       \draw[line width=0.25mm](\x,0+1.5) -- (\x,1+1.5);
       \draw[line width=0.25mm](\x,8) -- (\x,9);
       }
    \foreach \x in {1,2,8,9}{
       \draw[line width=0.25mm](\x,2+1.5) -- (\x,4+1.5);
       \draw[line width=0.25mm](\x,5) -- (\x,7);
       }

    \draw[line width=0.4mm](0.5,1+1.5) -- (4.5,1+1.5);
   \draw[line width=0.4mm](0.5,2+1.5) -- (4.5,2+1.5);
   \draw[line width=0.4mm](0.5,1+1.5) -- (0.5,2+1.5);
   \draw[line width=0.4mm](4.5,1+1.5) -- (4.5,2+1.5);
+1.5
   \draw[line width=0.4mm](5.5,1+1.5) -- (9.5,1+1.5);
   \draw[line width=0.4mm](5.5,2+1.5) -- (9.5,2+1.5);
   \draw[line width=0.4mm](5.5,1+1.5) -- (5.5,2+1.5);
   \draw[line width=0.4mm](9.5,1+1.5) -- (9.5,2+1.5);

   \draw[line width=0.4mm](5.5,7) -- (9.5,7);
   \draw[line width=0.4mm](5.5,8) -- (9.5,8);
   \draw[line width=0.4mm](5.5,7) -- (5.5,8);
   \draw[line width=0.4mm](9.5,7) -- (9.5,8);

   \draw[line width=0.4mm](0.5,7) -- (4.5,7);
   \draw[line width=0.4mm](0.5,8) -- (4.5,8);
   \draw[line width=0.4mm](0.5,7) -- (0.5,8);
   \draw[line width=0.4mm](4.5,7) -- (4.5,8);

   \draw[line width=0.25mm](4,7) to[bend right=90] (6,7);
   \draw[line width=0.25mm](3,6.5) to[bend right=90] (7,6.5);
   \draw[line width=0.25mm](3,6.5) -- (3,7);
   \draw[line width=0.25mm](7,6.5) -- (7,7);
   \draw[line width=0.25mm](3,2+1.5) -- (3,2.5+1.5);
   \draw[line width=0.25mm](7,2+1.5) -- (7,2.5+1.5);
   \draw[line width=0.25mm](4,2+1.5) to[bend left=90] (6,2+1.5);
   \draw[line width=0.25mm](3,2.5+1.5) to[bend left=90] (7,2.5+1.5);
   
   \draw[circle,fill](2.5,2.25+1.5)circle[radius=0mm]node[below]{  $f_{\bi}$};
   \draw[circle,fill](7.5,2.25+1.5)circle[radius=0mm]node[below]{  $f_{\bj}$};
   \draw[circle,fill](2.5,8.25)circle[radius=0mm]node[below]{  $f_{\bi}$};
   \draw[circle,fill](7.5,8.25)circle[radius=0mm]node[below]{  $f_{\bj}$};
   
   \draw[circle,fill](2.5,1+1.5)circle[radius=0mm]node[below]{  $\dots$};
   \draw[circle,fill](2.5,9)circle[radius=0mm]node[below]{  $\dots$};
   \draw[circle,fill](7.5,1+1.5)circle[radius=0mm]node[below]{  $\dots$};
   \draw[circle,fill](7.5,9)circle[radius=0mm]node[below]{  $\dots$};
   \draw[circle,fill](1.5,3.5+1.5)circle[radius=0mm]node[below]{  $\dots$};
   \draw[circle,fill](1.5,7)circle[radius=0mm]node[below]{  $\dots$};
   \draw[circle,fill](8.5,3.5+1.5)circle[radius=0mm]node[below]{  $\dots$};
   \draw[circle,fill](8.5,7)circle[radius=0mm]node[below]{  $\dots$};
   \draw[circle,fill](6.2,4+1.5)circle[radius=0mm]node[below]{  $\iddots$};
   \draw[circle,fill](6.1,7.4)circle[radius=0mm]node[below]{  $\ddots$};
   \draw[circle,fill](3.8,4+1.5)circle[radius=0mm]node[below]{  $\ddots$};
   \draw[circle,fill](3.9,7.4)circle[radius=0mm]node[below]{  $\iddots$};

   \draw[circle,fill](5,7.2)circle[radius=0mm]node[below]{$\text{k arcs}$};
   \draw[circle,fill](5,4.3+1.5)circle[radius=0mm]node[below]{$\text{k arcs}$};

 \end{tikzpicture}%
}

\newcommand\AkBasis{%
  \begin{tikzpicture}[scale = 0.55]
    \foreach \x in {1,4,6,9}{
       \draw[circle,fill] (\x,0)circle[radius=0.5mm]node[below]{};
       \draw[circle,fill] (\x,9)circle[radius=0.5mm]node[below]{};
       \draw[line width=0.25mm](\x,0) -- (\x,1);
       \draw[line width=0.25mm](\x,8) -- (\x,9);
       }
    \foreach \x in {1,2,8,9}{
       \draw[line width=0.25mm](\x,2) -- (\x,4);
       \draw[line width=0.25mm](\x,5) -- (\x,7);
       }
    \draw[circle,fill] (0,0)circle[radius=0mm]node[below]{};
    \draw[circle,fill] (10,0)circle[radius=0mm]node[below]{};

    \draw[line width=0.4mm](0.5,1) -- (4.5,1);
   \draw[line width=0.4mm](0.5,2) -- (4.5,2);
   \draw[line width=0.4mm](0.5,1) -- (0.5,2);
   \draw[line width=0.4mm](4.5,1) -- (4.5,2);

   \draw[line width=0.4mm](5.5,1) -- (9.5,1);
   \draw[line width=0.4mm](5.5,2) -- (9.5,2);
   \draw[line width=0.4mm](5.5,1) -- (5.5,2);
   \draw[line width=0.4mm](9.5,1) -- (9.5,2);

   \draw[line width=0.4mm](5.5,7) -- (9.5,7);
   \draw[line width=0.4mm](5.5,8) -- (9.5,8);
   \draw[line width=0.4mm](5.5,7) -- (5.5,8);
   \draw[line width=0.4mm](9.5,7) -- (9.5,8);

   \draw[line width=0.4mm](0.5,7) -- (4.5,7);
   \draw[line width=0.4mm](0.5,8) -- (4.5,8);
   \draw[line width=0.4mm](0.5,7) -- (0.5,8);
   \draw[line width=0.4mm](4.5,7) -- (4.5,8);

   \draw[line width=0.4mm](0.5,4) -- (9.5,4);
   \draw[line width=0.4mm](0.5,5) -- (9.5,5);
   \draw[line width=0.4mm](0.5,4) -- (0.5,5);
   \draw[line width=0.4mm](9.5,4) -- (9.5,5);

   \draw[line width=0.25mm](4,7) to[bend right=90] (6,7);
   \draw[line width=0.25mm](3,6.5) to[bend right=90] (7,6.5);
   \draw[line width=0.25mm](3,6.5) -- (3,7);
   \draw[line width=0.25mm](7,6.5) -- (7,7);
   \draw[line width=0.25mm](3,2) -- (3,2.5);
   \draw[line width=0.25mm](7,2) -- (7,2.5);
   \draw[line width=0.25mm](4,2) to[bend left=90] (6,2);
   \draw[line width=0.25mm](3,2.5) to[bend left=90] (7,2.5);

   \draw[circle,fill](2.5,2.25)circle[radius=0mm]node[below]{  $f_{\bi}$};
   \draw[circle,fill](7.5,2.25)circle[radius=0mm]node[below]{  $f_{\bj}$};
   \draw[circle,fill](2.5,8.25)circle[radius=0mm]node[below]{  $f_{\bi}$};
   \draw[circle,fill](7.5,8.25)circle[radius=0mm]node[below]{  $f_{\bj}$};
   
   \draw[circle,fill](5,5.6)circle[radius=0mm]node[below]{  $f_{\bi \cap^k \bj}$};
   
   \draw[circle,fill](2.5,1)circle[radius=0mm]node[below]{  $\dots$};
   \draw[circle,fill](2.5,9)circle[radius=0mm]node[below]{  $\dots$};
   \draw[circle,fill](7.5,1)circle[radius=0mm]node[below]{  $\dots$};
   \draw[circle,fill](7.5,9)circle[radius=0mm]node[below]{  $\dots$};
   \draw[circle,fill](1.5,3.5)circle[radius=0mm]node[below]{  $\dots$};
   \draw[circle,fill](1.5,7)circle[radius=0mm]node[below]{  $\dots$};
   \draw[circle,fill](8.5,3.5)circle[radius=0mm]node[below]{  $\dots$};
   \draw[circle,fill](8.5,7)circle[radius=0mm]node[below]{  $\dots$};
   \draw[circle,fill](6.2,4)circle[radius=0mm]node[below]{  $\iddots$};
   \draw[circle,fill](6.1,7.4)circle[radius=0mm]node[below]{  $\ddots$};
   \draw[circle,fill](3.8,4)circle[radius=0mm]node[below]{  $\ddots$};
   \draw[circle,fill](3.9,7.4)circle[radius=0mm]node[below]{  $\iddots$};

   \draw[circle,fill](5,7.2)circle[radius=0mm]node[below]{$\text{k arcs}$};
   \draw[circle,fill](5,4.3)circle[radius=0mm]node[below]{$\text{k arcs}$};

 \end{tikzpicture}%
}

\newcommand\CentralArcsA{
\begin{tikzpicture}[scale = 0.45]
\foreach \x in {4,6}{
       \draw[circle,fill] (\x,0)circle[radius=0.5mm]node[below]{};
       \draw[circle,fill] (\x,4-1.2)circle[radius=0.5mm]node[below]{};
       }
\draw[line width=0.25mm](4,0) -- (4,4-1.2);
   \draw[line width=0.25mm](6,0) -- (6,4-1.2);
   \draw[circle,fill](5,2.5)circle[radius=0mm]node[below]{  $\dots$};
   \end{tikzpicture}
   }

\newcommand\CentralArcsB{\begin{tikzpicture}[scale = 0.45]
\foreach \x in {8,9,10,11,12,13}{
       \draw[circle,fill] (\x,0)circle[radius=0.5mm]node[below]{};
       \draw[circle,fill] (\x,4-1.2)circle[radius=0.5mm]node[below]{};
       }
\draw[line width=0.25mm](8,0) -- (8,4-1.2);
   \draw[line width=0.25mm](9,0) -- (9,4-1.2);
   \draw[line width=0.25mm](12,0) -- (12,4-1.2);
   \draw[line width=0.25mm](13,0) -- (13,4-1.2);
   \draw[line width=0.25mm](10,0) to[bend left=90] (11,0);
   \draw[line width=0.25mm](10,4-1.2) to[bend right=90] (11,4-1.2);
   \draw[circle,fill](8.5,2.5)circle[radius=0mm]node[below]{  $\dots$};
   \draw[circle,fill](12.5,2.5)circle[radius=0mm]node[below]{  $\dots$};
   \end{tikzpicture}
   }
\newcommand\CentralArcsC{\begin{tikzpicture}[scale = 0.45]
\foreach \x in {8,9,10,11,12,13,14,15}{
       \draw[circle,fill] (\x,0)circle[radius=0.5mm]node[below]{};
       \draw[circle,fill] (\x,4-1.2)circle[radius=0.5mm]node[below]{};
       }
\draw[line width=0.25mm](8,0) -- (8,4-1.2);
   \draw[line width=0.25mm](9,0) -- (9,4-1.2);
   \draw[line width=0.25mm](14,0) -- (14,4-1.2);
   \draw[line width=0.25mm](15,0) -- (15,4-1.2);
   \draw[line width=0.25mm](11,0) to[bend left=90] (12,0);
   \draw[line width=0.25mm](11,4-1.2) to[bend right=90] (12,4-1.2);
   \draw[line width=0.25mm](10,0) to[bend left=90] (13,0);
   \draw[line width=0.25mm](10,4-1.2) to[bend right=90] (13,4-1.2);
   \draw[circle,fill](8.5,2.5)circle[radius=0mm]node[below]{  $\dots$};
   \draw[circle,fill](14.5,2.5)circle[radius=0mm]node[below]{  $\dots$};
   \end{tikzpicture}
   }

\newcommand\TraceTL{%
  \begin{tikzpicture}[scale = 0.7]
    \foreach \x in {1,2,3,4}{
       \draw[circle,fill] (\x,0)circle[radius=0.5mm]node[below]{};
       \draw[circle,fill] (\x,1.5)circle[radius=0.5mm]node[below]{};
       }
   \draw[line width=0.25mm](1,0) -- (1,0.25);
   \draw[line width=0.25mm](2,0) -- (2,0.25);
   \draw[line width=0.25mm](1,1.25) -- (1,1.5);
   \draw[line width=0.25mm](2,1.25) -- (2,1.5);
   \draw[line width=0.25mm](3,0) -- (3,1.5);
   \draw[line width=0.25mm](4,0) -- (4,1.5);
   \draw[line width=0.4mm](0.75,0.25) -- (2.25,0.25);
   \draw[line width=0.4mm](0.75,1.25) -- (2.25,1.25);
   \draw[line width=0.4mm](0.75,0.25) -- (0.75,1.25);
   \draw[line width=0.4mm](2.25,0.25) -- (2.25,1.25);
   \draw[circle,fill](1.5,1.25)circle[radius=0mm]node[below]{  $c$};
   \draw[circle,fill](1.5,0.5)circle[radius=0mm]node[below]{  $\dots$};
   \draw[circle,fill](1.5,2.0)circle[radius=0mm]node[below]{  $\dots$};
   \draw[circle,fill](3.5,0.5)circle[radius=0mm]node[below]{  $\dots$};
   \draw[circle,fill](3.5,2.0)circle[radius=0mm]node[below]{  $\dots$};
    
   \draw[line width=0.25mm](2,0) to[bend right=90] (3,0);
   \draw[line width=0.25mm](2,1.5) to[bend left=90] (3,1.5);
   \draw[line width=0.25mm](1,0) to[bend right=90] (4,0);
   \draw[line width=0.25mm](1,1.5) to[bend left=90] (4,1.5);
  \end{tikzpicture}%
}

\newcommand\Meander{%
  \begin{tikzpicture}[scale = 0.7]
    \foreach \x in {1,2,3,4,5,6}{
       \draw[circle,fill] (\x,0)circle[radius=0.5mm]node[below]{};
      \draw[line width=0.25mm](\x,0.3) to (\x,-0.3);
       }

      \draw[line width=0.25mm](0,0) -- (7,0); 
   \draw[line width=0.25mm](1,-0.3) to[bend right=90] (2,-0.3);
   \draw[line width=0.25mm](3,-0.3) to[bend right=90] (6,-0.3);
    \draw[line width=0.25mm](4,-0.3) to[bend right=90] (5,-0.3);

   \draw[line width=0.25mm](1,0.3) to[bend left=90] (6,0.3);
   \draw[line width=0.25mm](2,0.3) to[bend left=90] (5,0.3);
    \draw[line width=0.25mm](3,0.3) to[bend left=90] (4,0.3);
  \end{tikzpicture}%
}

\newcommand\MeanderLabelled{%
  \begin{tikzpicture}[scale = 0.7]
    \foreach \x in {1,2,3,4,5,6}{
       \draw[circle,fill] (\x,0)circle[radius=0.5mm]node[below]{};
      \draw[line width=0.25mm](\x,0.3) to (\x,-0.3);
       }

   \draw[circle,fill](0.5,-0.25)circle[radius=0mm]node[above]{  $i_1$};
   \draw[circle,fill](1.5,-0.25)circle[radius=0mm]node[above]{  $i_2$};
   \draw[circle,fill](2.5,-0.25)circle[radius=0mm]node[above]{  $i_3$};
   \draw[circle,fill](3.5,-0.25)circle[radius=0mm]node[above]{  $i_4$};
   \draw[circle,fill](4.5,-0.25)circle[radius=0mm]node[above]{  $i_5$};
   \draw[circle,fill](5.5,-0.25)circle[radius=0mm]node[above]{  $i_6$};
      \draw[line width=0.25mm](0,0) -- (7,0); 
   \draw[line width=0.25mm](1,-0.3) to[bend right=90] (2,-0.3);
   \draw[line width=0.25mm](3,-0.3) to[bend right=90] (6,-0.3);
    \draw[line width=0.25mm](4,-0.3) to[bend right=90] (5,-0.3);

   \draw[line width=0.25mm](1,0.3) to[bend left=90] (6,0.3);
   \draw[line width=0.25mm](2,0.3) to[bend left=90] (5,0.3);
    \draw[line width=0.25mm](3,0.3) to[bend left=90] (4,0.3);
  \end{tikzpicture}%
}

\newcommand\ProjectorThreeA{%
  \begin{tikzpicture}[scale = 0.75]
    \foreach \x in {1,1.75,2.5}{
     
       \draw[circle,fill] (\x+7,0)circle[radius=0.4mm]node[below]{};
       \draw[circle,fill] (\x+7,2)circle[radius=0.4mm]node[below]{};
       }
   \draw[line width=0.2mm](8,0) -- (8,2);
   \draw[line width=0.2mm](8.75,0) -- (8.75,2);
   \draw[line width=0.2mm](9.5,0) -- (9.5,2);

\draw[circle,fill](7.65,1.5)circle[radius=0mm]node[below]{  $i$};
\draw[circle,fill](8.375,1.5)circle[radius=0mm]node[below]{  $j$};
\draw[circle,fill](9.125,1.5)circle[radius=0mm]node[below]{  $i$};
\draw[circle,fill](9.85,1.5)circle[radius=0mm]node[below]{  $j$};
   
  \end{tikzpicture}%
}

\newcommand\ProjectorThreeB{%
  \begin{tikzpicture}[scale = 0.75]
    \foreach \x in {1,1.5,2.5}{
     
       \draw[circle,fill] (\x+7,0)circle[radius=0.4mm]node[below]{};
       \draw[circle,fill] (\x+7,2)circle[radius=0.4mm]node[below]{};
       }
   \draw[line width=0.2mm](8,0) -- (8,2);
   \draw[line width=0.25mm](8.5,2) to[bend right=90] (9.5,2);
   \draw[line width=0.25mm](8.5,0) to[bend left=90] (9.5,0);

\draw[circle,fill](7.65,1.5)circle[radius=0mm]node[below]{  $i$};
\draw[circle,fill](9,1.5)circle[radius=0mm]node[below]{  $j$};
\draw[circle,fill](9,0.4)circle[radius=0mm]node[below]{  $i$};
\draw[circle,fill](9,2.4)circle[radius=0mm]node[below]{  $i$};
   
  \end{tikzpicture}%
}

\newcommand\ProjectorThreeC{%
  \begin{tikzpicture}[scale = 0.75]
    \foreach \x in {1,2,2.5}{
     
       \draw[circle,fill] (\x+7,0)circle[radius=0.4mm]node[below]{};
       \draw[circle,fill] (\x+7,2)circle[radius=0.4mm]node[below]{};
       }
   \draw[line width=0.2mm](9.5,0) -- (9.5,2);
   \draw[line width=0.25mm](8,2) to[bend right=90] (9,2);
   \draw[line width=0.25mm](8,0) to[bend left=90] (9,0);

\draw[circle,fill](9.85,1.5)circle[radius=0mm]node[below]{  $j$};
\draw[circle,fill](8.5,1.5)circle[radius=0mm]node[below]{  $i$};
\draw[circle,fill](8.5,0.45)circle[radius=0mm]node[below]{  $j$};
\draw[circle,fill](8.5,2.45)circle[radius=0mm]node[below]{  $j$};
   
  \end{tikzpicture}%
}

\newcommand\ProjectorThreeD{%
  \begin{tikzpicture}[scale = 0.75]
    \foreach \x in {1,2,3}{
     
       \draw[circle,fill] (\x+7,0)circle[radius=0.4mm]node[below]{};
       \draw[circle,fill] (\x+7,2)circle[radius=0.4mm]node[below]{};
       }
   \draw[line width=0.2mm](10,0) -- (8,2);
   \draw[line width=0.25mm](9,2) to[bend right=90] (10,2);
   \draw[line width=0.25mm](8,0) to[bend left=90] (9,0);

\draw[circle,fill](8.15,1.5)circle[radius=0mm]node[below]{  $i$};
\draw[circle,fill](9.85,1.5)circle[radius=0mm]node[below]{  $j$};
\draw[circle,fill](8.5,0.4)circle[radius=0mm]node[below]{  $j$};
\draw[circle,fill](9.5,2.4)circle[radius=0mm]node[below]{  $i$};
   
  \end{tikzpicture}%
}

\newcommand\ProjectorThreeE{%
  \begin{tikzpicture}[scale = 0.75]
    \foreach \x in {1,2,3}{
     
       \draw[circle,fill] (\x+7,0)circle[radius=0.4mm]node[below]{};
       \draw[circle,fill] (\x+7,2)circle[radius=0.4mm]node[below]{};
       }
   \draw[line width=0.2mm](8,0) -- (10,2);
   \draw[line width=0.25mm](8,2) to[bend right=90] (9,2);
   \draw[line width=0.25mm](9,0) to[bend left=90] (10,0);

\draw[circle,fill](8.15,1.5)circle[radius=0mm]node[below]{  $i$};
\draw[circle,fill](9.85,1.5)circle[radius=0mm]node[below]{  $j$};
\draw[circle,fill](9.5,0.4)circle[radius=0mm]node[below]{  $i$};
\draw[circle,fill](8.5,2.45)circle[radius=0mm]node[below]{  $j$};
   
  \end{tikzpicture}%
}

\newcommand\ProjectorTwoA{%
  \begin{tikzpicture}[scale = 0.75]
    \foreach \x in {1,2}{
     
       \draw[circle,fill] (\x+7,0)circle[radius=0.4mm]node[below]{};
       \draw[circle,fill] (\x+7,2)circle[radius=0.4mm]node[below]{};
       }
   \draw[line width=0.2mm](8,0) -- (8,2);
   \draw[line width=0.2mm](9,0) -- (9,2);

\draw[circle,fill](7.65,1.5)circle[radius=0mm]node[below]{  $i$};
\draw[circle,fill](8.5,1.5)circle[radius=0mm]node[below]{  $j$};
\draw[circle,fill](9.35,1.5)circle[radius=0mm]node[below]{  $i$};
   
  \end{tikzpicture}%
}

\newcommand\ProjectorTwoB{%
  \begin{tikzpicture}[scale = 0.75]
    \foreach \x in {1,2}{
     
       \draw[circle,fill] (\x+7,0)circle[radius=0.4mm]node[below]{};
       \draw[circle,fill] (\x+7,2)circle[radius=0.4mm]node[below]{};
       }

\draw[circle,fill](8.5,2.5)circle[radius=0mm]node[below]{  $j$};
\draw[circle,fill](8.5,1.5)circle[radius=0mm]node[below]{  $i$};
\draw[circle,fill](8.5,0.4)circle[radius=0mm]node[below]{  $j$};
\draw[circle,fill](8.5,-0.5)circle[radius=0mm]node[below]{};

\draw[line width=0.25mm](8,2) to[bend right=90] (9,2);
\draw[line width=0.25mm](8,0) to[bend left=90] (9,0);
   
  \end{tikzpicture}%
}
\newcommand\ColouredExampleA{\begin{tikzpicture}[scale = 0.7]
    \fill[blue!20](6.5,0) rectangle (10.5,2);
    \fill[green!20] (7,0) -- (8,2) -- (9,2) -- (8,0);
    \fill[red!20] (9.5,0) ellipse (0.5 and 0.3);
    \fill[white] (9,-0.3) rectangle (10,0);
    \fill[white] (9, 2.3) rectangle (10,2);
    \foreach \x in {7,8,9,10}{
       \draw[circle,fill] (\x,0)circle[radius=0.5mm]node[below]{};
       }
    \foreach \x in {8,9}{
       \draw[circle,fill] (\x,2)circle[radius=0.5mm]node[below]{};
       }
    \foreach \x in {7,9}{
       \draw[circle,fill](\x + 0.5,2.5)circle[radius=0mm]node[below]{  $1$};
       }
    \foreach \x in {8}{
       \draw[circle,fill](\x + 0.5,2.5)circle[radius=0mm]node[below]{  $3$};
       }
    \foreach \x in {6,8,10}{
       \draw[circle,fill](\x + 0.5,0.5)circle[radius=0mm]node[below]{  $1$};
       }
    \foreach \x in {7}{
       \draw[circle,fill](\x + 0.5,0.5)circle[radius=0mm]node[below]{  $3$};
       }   
    \draw[circle,fill](9.5,0.5)circle[radius=0mm]node[below]{  $2$};       
   \draw[line width=0.25mm](7,0) -- (8,2);
   \draw[line width=0.25mm](8,0) -- (9,2);
   \draw[line width=0.25mm](9,0) to[bend left=90] (15-5,0);
  \end{tikzpicture}}

\newcommand\ColouredExampleB{\begin{tikzpicture}[scale = 0.7]
    \fill[blue!20](11.5,0) rectangle (15.5,2);
    \fill[red!20] (13,0) -- (14,2) -- (15,2) -- (14,0);
    \fill[white] (12,-0.3) rectangle (14,0);
    
    \foreach \x in {13,14}{
       \draw[circle,fill] (\x,0)circle[radius=0.5mm]node[below]{};
       }
    \foreach \x in {12,13,14,15}{
       \draw[circle,fill] (\x,2)circle[radius=0.5mm]node[below]{};
       }
    \foreach \x in {11,12,13,15}{
       \draw[circle,fill](\x + 0.5,2.5)circle[radius=0mm]node[below]{  $1$};
       }
    \foreach \x in {}{
       \draw[circle,fill](\x + 0.5,2.5)circle[radius=0mm]node[below]{  $3$};
       }
    \foreach \x in {12,14}{
       \draw[circle,fill](\x + 0.5,0.5)circle[radius=0mm]node[below]{  $1$};
       }
    \foreach \x in {}{
       \draw[circle,fill](\x + 0.5,0.5)circle[radius=0mm]node[below]{  $3$};
       }   
    \draw[circle,fill](13.5,0.5)circle[radius=0mm]node[below]{  $2$};
    \draw[circle,fill](14.5,2.5)circle[radius=0mm]node[below]{  $2$};
   
   \draw[line width=0.25mm](8+5,0) -- (9+5,2);
   \draw[line width=0.25mm](9+5,0) -- (10+5,2);
   \draw[line width=0.25mm](7+5,2) to[bend right=90] (8+5,2);    
  \end{tikzpicture}}

\newcommand\ColouredExampleC{\begin{tikzpicture}[scale = 0.7]
    \fill[blue!20](6.5,0) rectangle (10.5,2);
    \fill[blue!20](11.5-5,0+2) rectangle (15.5-5,2+2);
    \fill[green!20] (7,0) -- (8,2) -- (9,2) -- (8,0);
    \fill[red!20] (13-5,0+2) -- (14-5,2+2) -- (15-5,2+2) -- (14-5,0+2);
    \fill[red!20] (9.5,0) ellipse (0.5 and 0.3);
    \fill[white] (9,-0.3) rectangle (10,0);
    \foreach \x in {7,8,9,10}{
       \draw[circle,fill] (\x,0)circle[radius=0.5mm]node[below]{};
       }
    \foreach \x in {8,9}{
       \draw[circle,fill] (\x,2)circle[radius=0.5mm]node[below]{};
       }
    \foreach \x in {7,9}{
       \draw[circle,fill](\x + 0.5,2.2)circle[radius=0mm]node[below]{  $1$};
       }
    \foreach \x in {8}{
       \draw[circle,fill](\x + 0.5,2.2)circle[radius=0mm]node[below]{  $3$};
       }
    \foreach \x in {6,8,10}{
       \draw[circle,fill](\x + 0.5,0.5)circle[radius=0mm]node[below]{  $1$};
       }
    \foreach \x in {7}{
       \draw[circle,fill](\x + 0.5,0.5)circle[radius=0mm]node[below]{  $3$};
       }   
    \draw[circle,fill](13.5-5,2.8)circle[radius=0mm]node[below]{  $2$};
    \draw[circle,fill](9.5,0.5)circle[radius=0mm]node[below]{  $2$};
    \draw[circle,fill](14.5-5,4.5)circle[radius=0mm]node[below]{  $2$};\draw[circle,fill](11.5-5,4.5)circle[radius=0mm]node[below]{  $1$};\draw[circle,fill](12.5-5,4.5)circle[radius=0mm]node[below]{  $1$};\draw[circle,fill](13.5-5,4.5)circle[radius=0mm]node[below]{  $1$};\draw[circle,fill](15.5-5,4.5)circle[radius=0mm]node[below]{  $1$};\draw[circle,fill](12.5-5,2.7)circle[radius=0mm]node[below]{  $1$};\draw[circle,fill](14.5-5,2.7)circle[radius=0mm]node[below]{  $1$};
   
   \draw[line width=0.25mm](8+5-5,0+2) -- (9+5-5,2+2);
   \draw[line width=0.25mm](9+5-5,0+2) -- (10+5-5,2+2);
   \draw[line width=0.25mm](7+5-5,2+2) to[bend right=90] (8+5-5,2+2);
    
   \draw[line width=0.25mm](12-5,0) -- (13-5,2);
   \draw[line width=0.25mm](13-5,0) -- (14-5,2);
   \draw[line width=0.25mm](14-5,0) to[bend left=90] (15-5,0);  
  \end{tikzpicture}}


\begin{abstract}
Generalised Temperley--Lieb categories with regions labelled by elements of a commutative algebra were introduced by M.~Khovanov and the second author in [\emph{Pure Appl. Math. Q.} {\bf 19} (2023), no. 5]. We consider the case where the regions are labelled by colours, corresponding to a complete set of orthogonal idempotents of a semisimple commutative algebra. We determine when these generalised Temperley--Lieb categories are semisimple and find the direct sum decompositions of tensor products of simple objects. As the main tool we use two-variable versions of Chebychev polynomials and coloured Jones--Wenzl projectors. As a consequence, we prove a conjecture of M. Khovanov and the second author on Gram determinants and non-degeneracy of trace pairings for the associated Temperley--Lieb algebras with coloured regions. 
\end{abstract}

\maketitle

\section{Introduction}
\label{sec:intro}

Temperley--Lieb categories $\TL(d)$ are tensor categories with several applications to knot theory and mathematical physics \cites{KauffBook,KaufLins,DG}. The basic objects of $\TL(d)$ are denoted by $[n]$, for $n$ a non-negative integer, and morphisms from $[n]$ to $[m]$ are $\KK$-linear combinations of diagrams of crossingless matchings of $n$ top and $m$ bottom points. The composition rule evaluates each closed circle to the scalar $d$ in the ground field $\KK$. The endomorphism algebras of the objects $[n]$ in these categories are the Temperley--Lieb algebras $\TL_n(d)$ that originate from statistical mechanics \cite{TL} and were rediscovered in operator algebra by Jones \cites{Jones1,Jones2}, and realised in a diagrammatic description by Kauffman \cite{Kauf}. The Temperley--Lieb categories can be used to study endomorphisms of the fundamental representation of the quantum group $U_q(\mathfrak{sl}_2)$ \cite{FK}.

This article concerns a generalisation of Temperley--Lieb categories $\TL(A,\omega)$ that appeared in \cite{KL}*{Section~2.4}. Here, $A$ is a commutative algebra over a field $\KK$ and $\omega\colon A\to A$ is a $\KK$-linear map called the \emph{wrapping map}. Each region of a crossingless matching is labelled by an element of $A$. When composing, a closed circle with an interior label is evaluated by applying $\omega$ to the label. These Temperley--Lieb categories with labelled regions are again monoidal categories. 

In this article, we focus on the case of the semisimple algebra $A=\KK^\ell$, where it suffices to label each region by one of a complete set of $\ell$ distinct orthogonal idempotents. We think of these labels as \emph{colours} of the regions.
Here, $\omega=(\omega_{ij})$ corresponds to an $\ell\times \ell$-matrix and when composing two crossingless matching diagrams, the result is zero if regions with distinct colours meet.  The subcategory of $\TL(\KK^n,\omega)$ where colours in neighbouring regions are distinct already appeared in \cite{EL}, and in \cites{Eli,Haz} in the two-coloured case, in the study of Soergel bimodules.
We prove two results concerning these TL categories $\TL(\KK^n,\omega)$ with coloured regions.
\begin{enumerate}
    \item \Cref{thm: TL is semisimple} proves that $\TL(\KK^\ell, \omega)$ is semisimple if and only if none of the pairs of matrix entries $\omega_{ij},\omega_{ji}$ are zeros of certain two-variable Chebychev polynomials $U_n(x,y)$ that we introduce in \Cref{subsec:JW-recursion}. The isomorphism classes of simple objects are indexed by sequences $\bi=(i_1,\ldots, i_{n+1})\in \{1,\ldots,\ell\}^{n+1}$, for any non-negative integer $n\geq 0$.
    \item \Cref{thm:tensor-dec} determines the tensor product decomposition of the simple objects in the semisimple case. Namely, for colour sequences $\bi = (i_1, \dots, i_{n+1})$ and $\bj = (j_1, \dots, j_{m+1})$, let $r$ be the number of labels at the start of $\bj$ that match with labels at the end of $\bi$, so $j_1 = i_{n+1}$, $j_2 = i_n$, $\dots$, $j_r = i_{n+2-r}$ and $j_{r+1}\neq i_{n+1-r}$. If $r=0$ then $T_\bi\otimes T_\bj \cong 0$, and if $r>0$, then 
    \begin{equation*}
    T_\bi\otimes T_\bj = \bigoplus_{k = 0}^{r-1} T_{\bi \cap^k \bj},
    \end{equation*}
    for 
    $\bi \cap^k\bj = (i_1, \dots, i_{n+1-k}, j_{k+2}, \dots, j_{m+1}) = (i_1, \dots, i_{n-k}, j_{k+1}, \dots, j_{m+1})$ being the label sequence obtained by concatenating $\bi$ and $\bj$ and removing $2k+1$ overlapping entries.
\end{enumerate}
These results generalise well-known results for the classical (unlabelled) Temperley--Lieb categories
\begin{enumerate}
    \item That $\TL(d)$ is semisimple if and only if $d\neq q+q^{-1}$ for a root of unity $q$, i.e., when $d$ is not the root of a Chebychev polynomial of the second kind. 
    \item The tensor product decomposition, also called the Clebsch--Gordan rule, 
    $$ T_n\otimes T_m = T_{n+m}\oplus T_{n+m-2}\oplus \ldots \oplus T_{|n-m|}$$
    of the simple objects $\{ T_n~|~\text{for $n\geq 0$}\}$ in $\TL(d)$ in the semisimple case.
\end{enumerate}
A key tool in proofs are coloured Jones--Wenzl projectors whose images correspond to the simple objects of $\TL(\KK^\ell,\omega)$. These Jones--Wenzl projectors already appeared in \cite{EL}*{Section 2.6}, and \cites{Eli,Haz} in the two-colour case. A special case of the tensor product decompositions appeared in \cite{EL}*{Corollary 2.18}.

As an application, we show in \Cref{cor:conjecture23} that the Gram determinants of natural pairings on the generalised Temperley--Lieb algebras $\TL(\KK^\ell,\omega)$ are given by products of the two-variable Chebychev polynomials, evaluated at pairs $(\omega_{ij}, \omega_{ji})$ and we identify when these are non-zero. This result is used to prove \cite{KL}*{Conjecture 23}. As outlined in \Cref{sec:Gram-meander}, these Gram determinants have an interpretation as meander determinants with coloured regions, generalising the work of \cites{DFGG,DiFrancesco}.

We now summarise the structure of this article. We start by introducing the technical setup of Krull--Schmidt categories and techniques of filtrations on such categories in \Cref{sec:background}. We define the Temperley--Lieb categories with coloured regions and find their indecomposable objects in \Cref{sec:TL-labelled}. The results concerning semisimplicity and Jones--Wenzl projectors are contained in \Cref{sec:semisimple} and those on tensor product decompositions are contained in \Cref{sec:tensorproduct}. \Cref{app: chebychev} contains proofs of some elementary properties of the two-variable Chebychev polynomials $U_n(x,y)$ that may be of independent interest.

\subsection*{Acknowledgements}

This research was funded by a Nottingham Research Fellowship held by R.~L. This funding supported research internships for both of the other authors.  M.~R. is further supported by an EPSRC Postgraduate Studentship (Reference Number 2926264). R.~L. thanks Mikhail Khovanov and Emily Peters for interesting conversations that provided ideas used in this article. 

\section{Background}
\label{sec:background}

In this section, we collect relevant background material used in the main sections. 

\subsection{Krull--Schmidt categories}
\label{sec:KS}

Throughout this article, we say a category $\cC$ is \emph{additive} if it is enriched over the category of abelian groups and has finite direct sums. We further let $\KK$ denote an arbitrary field. We say that  category $\cC$ is \emph{$\KK$-linear} if $\cC$ is enriched over $\KK$-vector spaces and additive. The \emph{additive closure} of a category is denoted by $\cC^\oplus$. A category $\cC$ that is enriched over $\KK$-vector spaces is called \emph{hom-finite} if, for any two objects $X,Y$ in $\cC$, $\Hom_\cC(X,Y)$ is finite-dimensional. 

The \emph{idempotent completion} of an additive category $\cC$ is denoted by $\cC^\circ$. Objects in $\cC^\circ$ are pairs $(X,e)$ where $X$ is an object of $\cC$ and $e\in \End_\cC(X)$ is an idempotent endomorphism. A morphism $g\colon (X,e)\to (Y,f)$ in  $\cC^\circ$ corresponds to a morphism $g\colon X\to Y$ in $\cC$ such that $g=fge$. We define the \emph{Karoubian envelope} of $\cC$ to be the category $\Kar(\cC):= (\cC^\oplus)^\circ$. One shows that $\Kar(\cC)$ is both additive and idempotent complete. It is $\KK$-linear provided that $\cC$ is $\KK$-linear.

We say that an object $X$ of an additive category $\cC$ is called a \emph{Karoubi generator} of $\cC$ if $\cC = \Kar(X)$, where $\Kar(X)$ denotes the Karoubian envelope of the full subcategory of $\cC$ on the object $X$. 
By \cite{Krause}*{Proposition~2.3}, if we have a Karoubi generator $X$ of an additive and idempotent complete category $\cC$, so $\cC = \Kar(X)$, there is an equivalence of categories between $\cC$ and the category $\proj{\End_\cC(X)}$ of finitely generated projective right $\End_\cC(X)$-modules.

An additive category $\cC$ is \emph{Krull--Schmidt (KS)} if every object decomposes into a finite direct sum of objects which each have a local endomorphism ring. In a KS category, an object is indecomposable (i.e., does not decompose as a direct sum of non-zero objects) if and only if its endomorphism ring is local. It was shown in \cite{Krause}*{Section 4} that decompositions of object as a direct sum of indecomposable objects in a KS category are unique up to reordering and isomorphism of summands.  The following proposition gives an easy way to determine when a hom-finite category is Krull--Schmidt.

\begin{proposition}[\cite{FLP}*{Lemma 4.1}]\label[proposition]{prop:FLP1}
    Let $\cC$ be enriched over $\KK$-vector spaces and hom-finite. Then $\cC$ is Krull--Schmidt if and only if it is additive and idempotent complete. In particular, $\Kar(\cC)$ is Krull--Schmidt.
\end{proposition}

\subsection{Quotients and filtrations}
\label{sec:KS-filtrations}

We will use filtrations on KS categories to split them up into categories which are easier to understand. Following \cite{FLP}*{Section 4.3}, we say that a filtration of a KS category is an ascending chain $$\cC_0\subseteq \cC_1 \subseteq \dots \subseteq \cC$$ of full KS subcategories such that $\bigcup_{n\geq 0}\cC_n = \cC$. We also define $\cC_{-1} = \{0\}$ when it is needed.

For a full subcategory $\cD$ of a $\KK$-linear category $\cC$ we can define the quotient category $\cC/\cD$ as $\cC/\cI_\cD$ where $\cI_\cD$ is the ideal of $\cC$ spanned by $\{\id_X ~|~ X\in \cD\}$. We denote by $\pi:\cC \to \cC/\cD$ projection functor to the quotient, which is an additive $\KK$-linear functor.

\begin{lemma}[\cite{FLP}*{Lemma 4.3}]\label[lemma]{lem:indecomposables-in-quotient}
For a KS category $\cC$, let $\cI$ be an ideal of $\cC$ and let $\pi: \cC \to \cC/\cI$. Then $\cC/\cI$ is a Krull–Schmidt category. Additionally if $X$ is an indecomposable object in $\cC$ then either $\pi(X)\cong 0$ or $\pi(X)$ is indecomposable in $\cC/\cI$, and all indecomposables in $\cC/\cI$ are of this form. Finally if $X$ and $Y$ are indecomposable in $\cC$ and $\pi(X)\cong \pi(Y) \not\cong 0$ then $X\cong Y$.
\end{lemma}
A consequence of this lemma is that for any Karoubi generator $X$ of $\cC$, $\pi(X)$ is a Karoubi generator of the quotient. 

Now, \cite{FLP}*{Corollary 4.5, Lemma 4.13} shows that if $\cC$ is a KS category and $\cD$ is a full KS subcategory then 
\begin{equation}\label{eq:indec-quot}
\Indec(\cC) \cong \Indec(\cC/\cD) \sqcup \Indec(\cD),
\end{equation}
where $\cong$ denotes a bijection of sets and $\Indec(\cC)$ is the set of indecomposable objects of $\cC$. Applying this to a filtration $\cC_n$ of a KS category $\cC$ shows that
\begin{equation}\label{eq:indec}
\Indec(\cC)\cong \bigsqcup_{n\geq 0}\Indec(\cC_{n}/\cC_{n-1}).
\end{equation}

\subsection{Semisimple categories}

We say that a $\KK$-linear and hom-finite KS category $\cC$ is \emph{semisimple} if the endomorphism ring of each object is a semisimple algebra. We will require the following two lemmas concerning semisimple categories.

\begin{lemma}\label[lemma]{lem: object semisimple implies semisimple}
    If a $\KK$-linear hom-finite KS category $\cC$ has a set of disjoint linear simple objects $\{ X_i\}_{i\in I}$, meaning $\Hom_\cC(X_i, X_j) \cong \delta_{ij}\KK$, and every object of $\cC$ is isomorphic to a finite direct sum of copies of objects $X_i$, then $\cC$ is semisimple.
    \begin{proof}
        Let $Y$ be an object in $\cC$, so $Y\cong Y_1^{\oplus a_1} \oplus \dots \oplus Y_n^{\oplus a_n}$ for $n$ distinct simple objects $Y_1, \dots, Y_n \in \{ X_i\}_{i\in I}$ and $a_1, \dots, a_n \geq 1$. Endomorphisms of such an object are $a\times a$-matrices of morphisms $Y_i\to Y_j$, for $a=\sum_{i=1}^n a_i$. By assumption, all entries correspond to scalars and are zero for non-isomorphic $Y_i$'s. Thus, $\End_\cC(Y)\cong M_{a_1 \times a_1}(\KK) \oplus \dots \oplus M_{a_n \times a_n}(\KK)$, which by the Artin--Wedderburn theorem (\cite{assem}*{Theorem 3.4}) is a semisimple algebra, so $\cC$ is semisimple.
    \end{proof}
\end{lemma}
\begin{lemma}\label[lemma]{lem: no central nilpotent in semisimple algebra}
    A semisimple finite-dimensional $\KK$-algebra contains no non-zero central nilpotent elements.
    \begin{proof}
        By the Artin--Wedderburn theorem a semisimple finite-dimensional $\KK$-al\-ge\-bra is isomorphic to a finite direct product of matrix rings $M_1 \times \dots \times M_n$ with entries in division rings. Let $x = (x_1, \dots, x_n)$ be a central nilpotent element in $M_1 \times \dots \times M_n$, so $x\neq 0$ and there exists $k\in \NN$ such that $x^k=0$. Then each of the components $x_k$, $1\leq k\leq n$ is a central nilpotent element in the matrix ring $M_k$ and hence zero. Thus, $x=0$. 
    \end{proof}
\end{lemma}

\section{Temperley--Lieb categories with labelled regions}
\label{sec:TL-labelled}

In this section, we construct Temperley--Lieb (TL) categories with labelled regions following \cite{KL}*{Sections~2.4 and 3.2}. 
We fix the input datum of a pair $(A,\omega)$, where $A$ is a commutative $\KK$-algebra $A$ and  $\omega: A \to A$ is a $\Bbbk$-linear map called the \emph{wrapping map}.

\subsection{Definition of labelled Temperley--Lieb categories}
\label{sec:TL-definition}

    We first construct the category $\TL^0(A, \omega)$ which has:
    \begin{itemize}
        \item     Objects $[n]$, for each non-negative integer $ n \in \mathbb{Z}_{\geq 0}$.
        \item The morphisms between $[n]$ and $[m]$ in this category are formal $\KK$-linear combinations of labelled crossingless matchings with $n$ points at the top and $m$ points at the bottom. 
    \end{itemize}
 These labelled crossingless matchings have each of their regions labelled by an element of $A$ and the linear combinations are $\KK$-multilinear in each region, for example if $a,b,c,d,e \in A$ and $\lambda \in \KK$, then
    \begin{equation}
    \vcenter{\hbox{\TLMultilinear}}
    \end{equation}
    Composition of these morphisms is done by vertically stacking the diagrams. Wherever any two regions meet in this stacking the labels are multiplied together as elements of $A$, and wherever a closed loop appears we apply the wrapping map $\omega$ to the interior label. An example of such a composition is given, diagrammatically, as follows:
    \begin{equation}
    \vcenter{\hbox{\resizebox{0.9\textwidth}{!}{\TLMultiply}}}
    \end{equation}
The multilinearity of labels means that for any chosen $\KK$-basis of $A$, any linear combination of labelled diagrams can be rearranged to a linear combination of diagrams only labelled by basis vectors. 

For any $n$ and $m$ there are only finitely many crossingless matchings with $n$ points at the top and $m$ points at the bottom, so if $A$ is finite-dimensional, then $\TL^0(A, \omega)([n],[m])$ is finite dimensional and $\TL^0(A, \omega)$ is hom-finite and we can take the Karoubian envelope to get a $\KK$-linear category.\footnote{In \cite{KL}, these categories are denoted by $\mathbb{SU}_{\omega}$ for $\TL^0(A,\omega)$ and $\mathbb{DSU}_{A,\omega}$ for its Karoubian envelope $\TL(A,\omega)$.}  

\begin{definition}\label[definition]{def:TL}
We define the \emph{Temperley--Lieb (TL) category with labelled regions} $\TL(A,\omega)$ to be the Karoubian envelope of $\TL^0(A,\omega)$.
\end{definition}

By \Cref{prop:FLP1} we have that $\TL(A,\omega)$ is Krull--Schmidt since $\TL^0(A, \omega)$ is hom-finite. We can view $\TL^0(A,\omega)$ as a subcategory of $\TL(A,\omega)$ by identifying $[n] \in \TL^0(A,\omega)$ with $([n], 1_{[n]}) \in \TL(A,\omega)$. The morphism $1_{[n]}$, the identity of $[n]$ in $\TL^0(A,\omega)$, is simply the identity diagram with each region labelled by $1_A$, the identity of the algebra $A$. Contained in the category $\TL(A,\omega)$ are analogues of TL algebras with labelled regions.

\begin{definition}
    For $n\geq 0$, the \emph{Temperley--Lieb (TL) algebra with labelled regions}  is the algebra $\TL_n(A,\omega)$ of endomorphisms of the object $[n]$ in the category $\TL(A, \omega)$.
\end{definition}

Setting $A=\KK$, $\omega$ is given by multiplication with a scalar $d$. In this case, $\TL(\KK,d)$ recovers the classical TL category and $\TL_n(\KK,d)$ the classical TL algebra, see \cites{Jones1,Jones2,Kauf}.

\subsection{The case of a semisimple algebra of labels}
\label{sec:TL-semisimple-algebra}

We will now focus on the specific case of the semisimple $\KK$-algebra $A = \KK^\ell$ spanned over $\KK$ by the standard basis vectors $e_1, \dots, e_\ell$. Note that the $e_i$ are the primitive idempotents and $1_A=e_1+\ldots +e_\ell$. Because of multilinearity we can just use these $e_1, \dots, e_\ell$ for the labels instead of using all the elements of $A$, without any loss of generality. The wrapping map $\omega \colon \KK^\ell \to \KK^\ell$ can be represented by the $\ell\times \ell$ matrix $(\omega_{ij})$ such that 
\begin{equation}\label{eq:omega-matrix-expression}
\omega(e_j)=\sum_{i=1}^\ell \omega_{ji}e_i.
\end{equation}
With this algebra $A=\Bbbk^\ell$, we interpret the labels of the regions by the orthogonal idempotents as $\ell$ distinct \emph{colours} and will refer to the generalised TL category $\TL(\KK^\ell,\omega)$ as having \emph{coloured regions} in the following. For the colours here we will just write them as $1$, $2,\ldots$ rather than $e_1$, $e_2,\ldots$ , or $i_1$, $i_2,\ldots$ rather than $e_{i_1}$, $e_{i_2},\ldots$, for simplicity. The composition rule implies that when two regions with distinct colours meet, the result is zero. For example, 
\begin{equation}
   \vcenter{\hbox{\resizebox{!}{5\baselineskip}{\ColouredExampleA }}} \circ \vcenter{\hbox{\resizebox{!}{5\baselineskip}{\ColouredExampleB }}} = \vcenter{\hbox{\resizebox{!}{8\baselineskip}{\ColouredExampleC }}}=0.
\end{equation}
In the following, we will only use colours in explicit examples, leaving general diagrams in black and white. 
\begin{remark}\label[remark]{rem:EL-version}
    This category $\TL(\KK^\ell, \omega)$ is similar to the 2-category constructed in \cite{EL}*{Section 2.6} called $S\mathcal{TL}(\omega)$, the difference being that in $\TL(\KK^\ell, \omega)$ we allow diagrams to have consecutive regions labelled by the same colour.
\end{remark}

For a sequence of colours $\bi = (i_1, \dots, i_{n+1})$ of length $n+1$, where $i_1,i_2,\ldots\in \{1,\ldots, \ell\}$, we denote
\begin{equation}
1_{\bi}:=\vcenter{\hbox{\MatchingOnes}}.
\end{equation}
We note that $1_{\bi}$ is an idempotent of $[n]$ in $\TL^0(\KK^\ell, \omega)$, so $$[\bi]:=([n], 1_{\bi})$$ is an object in $\TL(\KK^\ell,\omega)$. Let $\bj = (j_1,\dots, j_{m+1})$, then 
$$\Hom_{\TL(\KK^\ell,\omega)}([\bi],[\bj]) = 1_{\bj}\Hom_{\TL^0(\KK^\ell,\omega)}([n],[m])1_{\bi}$$
is $\{0\}$ unless $n+m$ is even, $i_1 = j_1$ and $i_{n+1} = j_{m+1}$. If these three conditions all hold then $\Hom_{\TL(\KK^\ell,\omega)}([\bi],[\bj])$ is spanned by formal linear combinations of crossingless matchings with colour sequence $\bi$ along the top, and $\bj$ along the bottom. 

The endomorphisms $\End_{\TL(\KK^\ell,\omega)}([\bi])$ again form a $\KK$-algebra, which we call the generalised TL algebra $\TL_{\bi}(\KK^\ell, \omega)$. Matchings of the shape 
\begin{equation}    h_k^\bi:=\vcenter{\hbox{\resizebox{0.75\textwidth}{!}{\TLGeneratorh}}}
\label{eq:hkbi}
\end{equation}
for $1\leq k \leq n-1$ such that $i_k=i_{k+2}$ can be composed to obtain any other crossingless matching, up to scalar, (cf. \cite{DFGG}*{Section 3} in the classical case). Since all matchings in $\TL_{\bi}(\KK^\ell, \omega)$ have the same colour sequence $\bi$ at the top and bottom these diagrams $h_k^\bi$ for $1\leq k \leq n-1$ must generate $\TL_{\bi}(\KK^\ell, \omega)$.

\subsection{Filtrations on coloured Temperely--Lieb categories}
\label{sec:filtration}

Next we introduce a filtration on the category $\cC = \TL(\KK^\ell,\omega)$. We define $\cC_n$ to be the Karoubian envelope of the full subcategory $\{[\bi] ~|~ \bi \text{ has length }\leq n+1\}$. 

For any $n\in \ZZ_{\geq 0}$,  the object $[n]=([n],1_{[n]})$  in $\TL(\KK^\ell,\omega)$ is isomorphic to the direct sum of all the objects $[\bi]$ where $\bi$ has length $n+1$. Thus, $\cC_n = \Kar(\{[k]~|~0 \leq k\leq n\})$ and $\bigoplus_{k = 0}^n [k]$ is a Karoubi generator of $\cC_n$. Note that $\cC_{-1}$ is defined as $0$. Every object $([n],e)$ in $\TL(\KK^\ell,\omega)$ will be in the $n$th part of the filtration for some $n$. We can therefore find the indecomposable objects in $\cC=\TL(\KK^\ell,\omega)$ by finding the indecomposable objects in each $\cC_n/\cC_{n-1}$ as explained in \Cref{sec:KS-filtrations}.

Applying that the quotient functor preserves Karoubi generators and \cite{Krause}*{Proposition~2.3} we see that 
$$\cC_n/\cC_{n-1} = \Kar(\pi(\textstyle\bigoplus_{k = 0}^n [k]))\simeq \proj{\End_{\cC_n/\cC_{n-1}}(\pi(\textstyle\bigoplus_{k = 0}^n [k]))},$$where $\pi$ is the projection map from $\cC_n$ to $\cC_n/\cC_{n-1}$, so we now look to simplify this algebra $\End_{\cC_n/\cC_{n-1}}(\pi(\bigoplus_{k = 0}^n [k]))$. First we can see that if an endomorphism factors through an object in $X \in \cC_{n-1}$ then this endomorphism in $0$ in $\cC_n/\cC_{n-1}$ because $\id_X$ is $0$. If we consider the two-sided ideal $\cI_{n-1}$ in $\TL_n(A,\omega)$ defined by  
$$\cI_{n-1}(X,Y):=\left\{f:X \to Y ~|~ f\text{ factors through objects in }\cC_{n-1}\right\},$$ then
\begin{align*}
    \End_{\cC_n/\cC_{n-1}}(\pi(\textstyle\bigoplus_{k = 0}^n [k]))&=\faktor{\textstyle\End_{\cC_n}(\bigoplus_{k = 0}^n [k])}{\cI_{n-1}(\textstyle\bigoplus_{k = 0}^n [k],\textstyle\bigoplus_{k = 0}^n [k])}\\
    &\cong\faktor{\End_{\cC_n}([n])}{\cI_{n-1}([n],[n])}.
\end{align*}
The  isomorphism of algebras $\cong$ follows from the observation that $\id_{[k]}$, with $k<n$, is contained in $\cI_{n-1}$ and hence zero in the quotient.

We can now use that $[n] = \bigoplus_{\bi \in I}[\bi]$ to simplify the algebra further. In ${\End_{\cC}([\bi])}$, every diagram except the identity diagram will factor through an object in $\cC_{n-1}$. For example, for the sequence $\bi = (1,2,1,2,1)$:
\begin{align}\label{eqn: non identity factor}
&\vcenter{\hbox{\resizebox{!}{4\baselineskip}{\TLNonIdentityFactorA}}}\,\,\, = \vcenter{\hbox{\resizebox{!}{4\baselineskip}{\TLNonIdentityFactorB}}} \circ \vcenter{\hbox{\resizebox{!}{4\baselineskip}{\TLNonIdentityFactorC}}}\\
  &\hspace{34pt}[\bi]\to[\bi]\hspace{70pt} [\bi]\to[(1,2,1)]\hspace{48pt} [(1,2,1)]\to[\bi]\nonumber
\end{align}
For any two distinct colour sequences $\bi$ and $\bj$ of length $n+1$, $\Hom_{\cC_n/\cC_{n-1}}([\bi],[\bj]) = 0$ because any diagram except the identity factors through an object in $\cC_{n-1}$, and in the identity diagram the colours will not all match because $[\bi]\neq[\bj]$. Therefore,  \begin{align*}
    \faktor{\End_{\cC}([n])}{\cI_{n-1}([n],[n])} &\cong \faktor{\End_{\cC}(\bigoplus_{\bi \in I}[\bi])}{\cI_{n-1}([n],[n])}\\
    &\cong \faktor{\left(\prod_{\bi \in I}\End_\cC([\bi])\right)}{\cI_{n-1}([n],[n])}\\
    &\cong \prod_{\bi \in I}\left(\faktor{\End_\cC([\bi])}{\cI_{n-1}([\bi],[\bi])}\right)=\KK^{|I|} = \KK^{\ell^{\,n+1}}.
\end{align*}
This simplification leads to the following result.
\begin{proposition}\label[proposition]{prop: filtration quotient is K^l^n+1 mod}
The quotient category $\cC_n/\cC_{n-1}$ is equivalent to the semisimple category  of finite-dimensional modules over $\KK^{\ell^{\,n+1}}$.
\begin{proof}
As $\End_{\cC_n/\cC_{n-1}}(\pi(\textstyle\bigoplus_{k = 0}^n [k]))\cong \KK^{\ell^{\,n+1}}$, we have $\cC_n/\cC_{n-1} \simeq \proj{\KK^{\ell^{\,n+1}}}$. Also $\KK^{\ell^{\,n+1}}$ is semisimple so all modules over this algebra are projective, and therefore $\cC_n/\cC_{n-1}$ is the category of finitely generated modules (and equivalently, finite-dimensional modules) over $\KK^{\ell^{\,n+1}}$, which is semisimple.
\end{proof}
\end{proposition}

We have seen above that $\pi([\bi])$ is a simple object in $\cC_n/\cC_{n-1}$ for any colour sequence $\bi$ of length $n+1$. Moreover, all the $\ell^{n+1}$ non-isomorphic simple objects in $\cC_n/\cC_{n-1}$ are of this form.

\begin{proposition}\label[proposition]{prop:bijection of indecomposables}
    For every $\bi$ of length $n+1$, there exists a unique indecomposable object $X_\bi \in \cC_n$ such that $X_\bi \notin \cC_{n-1}$ and $\pi(X_\bi) = \pi([\bi])$.
    \begin{proof}

        By \Cref{lem:indecomposables-in-quotient}, every indecomposable object $\pi([\bi])$ in $\cC_n/\cC_{n-1}$ is of the form $\pi(X_\bi)$ for some $X_\bi\in \Indec(\cC_n)$, and $X_\bi$ cannot be in $\cC_{n-1}$ otherwise we would have $\pi(X_\bi)\cong 0$, so $\pi(X_\bi)\in\Indec(\cC_n) \setminus \Indec(\cC_{n-1})$. Also, we have that $$\Indec(\cC_n) \setminus \Indec(\cC_{n-1})\cong \Indec(\cC_n/\cC_{n-1}) = \{\pi([\bi]) ~|~ \bi \text{ has length }n+1\}.$$ Finally if $X_\bi = X_\bj$ then $\pi([\bi]) = \pi(X_\bi)= \pi(X_\bj) = \pi([\bj])$ and $\bi = \bj$, so the assignment $\bi \mapsto X_\bi$ is a bijection between the set of colour sequences of length $n+1$ and  $\Indec(\cC_n) \setminus \Indec(\cC_{n-1})$.
    \end{proof}
\end{proposition}
\begin{corollary}\label[corollary]{cor: indec of C are Xis}
    The indecomposable objects in $\cC =\TL(\KK^\ell, \omega)$ are the $X_\bi$, for all colour sequences $\bi$ of length $\geq 0$.
    \begin{proof}
        This follows from the fact that $\cC = \bigcup_{n\geq 0}C_n$, see \eqref{eq:indec}, and at each step of the filtration $\Indec(\cC_n) \setminus \Indec(\cC_{n-1}) = \{X_\bi ~|~ \bi\text{ has length }n+1\}$.
    \end{proof}
\end{corollary}
In \Cref{sec:semisimple}, we will investigate when the objects $X_\bi$ are simple in order to prove semisimplicity of $\TL(\KK^\ell,\omega)$. Before this, we review the monoidal structure of this category.

\subsection{Monoidal structure}

The category $\TL(\KK^\ell,\omega)$ a strict monoidal category\footnote{This is a special case of the more general class of monoidal categories $\mathbb{SU}_\omega$ of \cite{KL}*{Section 2.4}.}. The monoidal product is defined by horizontal stacking of morphisms. 
\begin{definition}
On $\TL^0(\KK^\ell,\omega)$ we define $[n]\otimes [m] = [n+m]$, and if $f:[n]\to[m]$ and $g:[p]\to[q]$ then $f\otimes g:[n+p]\to [m+q]$ is formed by stacking the diagrams horizontally and extending this bilinearly. The tensor product extends to $\cC = \TL(\KK^\ell,\omega)$ by defining $([n],e)\otimes([n'],e') = ([n+n'], e\otimes e')$.

For example, consider the following tensor product of morphisms:
\begin{equation*}
   \vcenter{\hbox{\resizebox{!}{5\baselineskip}{\TLMonoidalProductA }}} \otimes \vcenter{\hbox{\resizebox{!}{5\baselineskip}{\TLMonoidalProductB }}} = \vcenter{\hbox{\resizebox{!}{5\baselineskip}{\TLMonoidalProductC }}}
\end{equation*}
We see that $f\otimes g$ is only non-zero if the last colour of $f$ is the same as the first colour of $g$. In this example, both colours are $2$ so the tensor product is non-zero. 
\end{definition}

The monoidal unit $\one$ of $\TL(\KK^\ell,\omega)$ is the object $[0]$ and admits a decomposition
\begin{equation}\label{eq:unit-dec}
\one = [(1)]\oplus \ldots \oplus [(\ell)].
\end{equation}

\begin{definition}\label[definition]{def:concat}
For $\bi = (i_1, \dots, i_{n+1})$ and $\bj = (j_1, \dots, j_{m+1})$, if $i_{n+1} = j_1$ then 
    we denote the \emph{concatenation} of $\bi$ and $\bj$ by
    \begin{equation*}
    \bi\bj := (i_1, \dots, i_{n+1}, j_2, \dots, j_{m+1}).
    \end{equation*}
\end{definition}
With this notation, we observe that 
$[\bi]\otimes [\bj]=[\bi\bj]$ provided that the last entry of the sequence $\bi$ matches the first entry of the sequence $\bj$.

\section{Semisimplicity of Temperley--Lieb categories with coloured regions}
\label{sec:semisimple}

This section contains the first main result on semisimplicity of the categories $\TL(A, \omega)$ for the semisimple algebras $A=\KK^\ell$ from \Cref{sec:TL-semisimple-algebra}.

\subsection{Classical Jones--Wenzl projectors}
    We will first review the Jones--Wenzl (JW) projectors for the classical (unlabelled) TL algebras $\TL_n(d)=\TL_n(\KK,d)$, as defined in \cite{Wen}, see also \cites{FK,Mor}. These are non-zero idempotents in $\TL_n(d)$, denoted $f_n$, which are uniquely defined for each $n\geq 1$. 

Here, $U_n(x)$ denotes the $n$th Chebychev polynomial of the second kind.\footnote{Note that these Chebychev polynomials of the second kind are often defined with a different normalisation which matches the one used here substituting $2x$ for $x$, cf. \cite{mason-handscome}*{Section 1.2.2}.} These polynomials are defined recursively by setting $U_{0}(x)=1$, $U_{1}(x)=x$ and $U_{n}(x)=xU_{n-1}(x) - U_{n-2}(x)$.

\begin{definition}\label[definition]{def:classical-JW}
Assume that $d$ is not a zero of any of the Chebychev polynomials $U_n(x)$.
The \emph{Jones--Wenzl (JW) projectors} $f_n\in \TL_n(d)$ are defined recursively by setting $f_0=1_{[0]}$, $f_{1}=1_{[1]}$, and, for $n\geq 1$
\begin{equation}
f_n \; = \;  \resizebox{!}{3\baselineskip}{$\vcenter{\hbox{\MatchingBoxT}}$}
\;-\; \frac{U_{n-2}(d)}{U_{n-1}(d)}    \resizebox{!}{3\baselineskip}{$\vcenter{\hbox{\MatchingBoxL}}$}\quad.
\end{equation}
\end{definition}

The JW projectors $f_n$ produced from this recurrence are  idempotents. They also have the property that adding a `cup' or `cap' to any JW projector gives zero, as described in \cite{Mor}. This is equivalent to multiplying on the left or right by $h_k$, the unlabelled version of the generators $h_k^\bi$, giving zero. \vspace{-5pt}
\begin{equation}
\vcenter{\hbox{\resizebox{0.7\textwidth}{!}{\MatchingBoxes}}}
\end{equation}

\subsection{Definition of the coloured Jones--Wenzl projectors}\label{subsec:JW-recursion}

We will now construct JW projectors for the coloured TL algebras $\TL_{\bi}(\KK^\ell, \omega)$ defined in \Cref{sec:TL-semisimple-algebra}. These coloured JW projectors will also be recursively defined non-zero idempotentes that have the same properties of being annihilated by adding `caps' and `cups', i.e., by multiplying with the generators $h_k^\bi$ of \eqref{eq:hkbi}, as the classical JW projectors. 

To define the coloured JW projectors we will use the following Chebychev polynomials.

\begin{definition}\label[definition]{def:two-var-Cheb}
We define \emph{two-variable Chebychev polynomials} $U_n(x,y)\in \ZZ[x,y]$, for $n\geq 0$, recursively by 
\begin{align}
    U_{0}(x,y)&=1, \\
    U_{1}(x,y)&=y,\\
    U_n(x,y)&=\begin{cases} xU_{n-1}(x,y) - U_{n-2}(x,y), &\text{if $n$ is even}, \\
    yU_{n-1}(x,y) - U_{n-2}(x,y), &\text{if $n$ is odd}.
    \end{cases}
\end{align} 
\end{definition}

\begin{example} Using the recursion, we compute the next entries in the series:
\begin{align*}
     U_{2}(x,y)&=xy-1,& U_{5}(x,y)&=x^2y^3-4xy^2+3y,\\
U_{3}(x,y)&=xy^2-2y,&  U_{6}(x,y)&=x^3y^3-5x^2y^2+6xy-1,\\
U_{4}(x,y)&=x^2y^2-3xy+1, & U_{7}(x,y)&=x^3y^4-6x^2y^3+10xy^2-4y.
\end{align*}
\end{example}

Setting $x=y$ recovers the classical Chebychev polynomials of the second kind which were used in \Cref{def:classical-JW}. Some properties of these two-variable Chebychev polynomials are discussed in \Cref{app: chebychev}.

\begin{remark}
    These two-variable Chebychev polynomials are related to the \emph{two-coloured quantum numbers} of \cites{EL}. We see that $[n]_{s,t} = (-1)^{n-1}U_{n-1}(a_{ts},a_{st})$ and $[n]_{t,s} = (-1)^{n-1}U_{n-1}(a_{st},a_{ts})$. See also \cites{Haz}. 
\end{remark}

In the following, we will define coloured JW projectors for all sequences $\bi$ of colours. We start by constructing JW projectors for alternating sequences $\bi=(i,j,\ldots)$ of length $n+1$ in the following definition before giving the general construction. Recall that the wrapping map $\omega$ is determined by the matrix $(\omega_{ij})$, see \eqref{eq:omega-matrix-expression}.

\begin{definition}\label[definition]{def:JW-proj}
We first define $f_0^{ij}=1_i$ and $f_1^{ij} = 1_{ij}$. Assuming that $f_{n-1}^{ij}$ has been defined and that $U_{n-1}(\omega_{ij},\omega_{ji}) \neq 0$, the \emph{coloured Jones--Wenzl (JW) projector} $f_n^{ij}$ on the object $\bi = (i,j,\dots)$ is defined by the formulas  
\begin{align}
f_n^{ij}\; &= \; \vcenter{\hbox{\resizebox{!}{5\baselineskip}{\MatchingBoxLOT }}} - \;\frac{U_{n-2}(\omega_{ij}, \omega_{ji})}{U_{n-1}(\omega_{ij}, \omega_{ji})}  \vcenter{\hbox{\resizebox{!}{5\baselineskip}{\MatchingBoxLOL }}}, \quad \text{when $n$ is odd,} \\
f_n^{ij}\; &=\; \vcenter{\hbox{\resizebox{!}{5\baselineskip}{\MatchingBoxLET }}} - \;\frac{U_{n-2}(\omega_{ij}, \omega_{ji})}{U_{n-1}(\omega_{ij}, \omega_{ji})}  \vcenter{\hbox{\resizebox{!}{5\baselineskip}{\MatchingBoxLEL }}}, \quad \text{when $n$ is even.}
 \end{align}
\end{definition}
In order to ensure that all JW projectors $f_n^{ij}$ exist, we will make the following assumption about the wrapping map $\omega$.

\begin{assumption}\label[assumption]{assumption:chebychev-non-zero}
We assume that the wrapping map $\omega: \KK^\ell \to \KK^\ell$, represented by the matrix $(\omega_{ij})$, satisfies $$U_n(\omega_{ij}, \omega_{ji})\neq 0$$ for all $1\leq i,j \leq \ell$ and all $n\geq 1$.   
\end{assumption}
\begin{remark}
    \Cref{assumption:chebychev-non-zero} is equivalent to the assumption in \cite{EL}*{Section 2.6} that, for all $k\geq 1$ and $m\geq 0$, the two-coloured quantum binomial coefficients $\big[{k\atop m}\big]_{s,t}=\frac{[1]_{s,t}[2]_{s,t}\ldots[k]_{s,t}}{[1]_{s,t}\ldots[m]_{s,t}[1]_{s,t}\ldots[k-m]_{s,t}}$, where $[n]_{s,t} = (-1)^{n-1}U_{n-1}(\omega_{ts},\omega_{st})$, are well-defined (after cancellation) and invertible. See also \cite{Haz} in the two-colour case with coefficients in a commutative ring.  
\end{remark}

\begin{example} Using the formulas, we can compute the JW projectors $f_2^{ij}$ and $f_3^{ij}$:
\begin{align}
f_2^{ij}\; &= \; \vcenter{\hbox{\resizebox{!}{3.25\baselineskip}{\ProjectorTwoA }}} - \;\frac{ 1}{\omega_{ji}}   \vcenter{\hbox{\resizebox{!}{4\baselineskip}{\ProjectorTwoB }}}\\
f_3^{ij}\; &=\; \vcenter{\hbox{\resizebox{!}{3.25\baselineskip}{\ProjectorThreeA }}} - \;\frac{\omega_{ji}}{\omega_{ij} \omega_{ji}-1}  \vcenter{\hbox{\resizebox{!}{3.5\baselineskip}{\ProjectorThreeB }}} -\frac{\omega_{ij}}{\omega_{ij} \omega_{ji}-1}  \vcenter{\hbox{\resizebox{!}{3.5\baselineskip}{\ProjectorThreeC }}}\\
 &+ \;\frac{ 1}{\omega_{ij}\omega_{ji}-1}   \vcenter{\hbox{\resizebox{!}{3.5\baselineskip}{\ProjectorThreeD}}}+ \;\frac{ 1}{\omega_{ij}\omega_{ji}-1}   \vcenter{\hbox{\resizebox{!}{3.5\baselineskip}{\ProjectorThreeE}}}\notag
 \end{align}
\end{example}

We now show that the $f_n^{ij}$ of \Cref{def:JW-proj} have analoguous properties to the classical JW projectors.

\begin{lemma}\label[lemma]{lem: coeff of identity diagram is 1}
The coefficient of $1_\bi$ in coloured JW projector $f_n^{ij}$ is $1$.
    \begin{proof}
This follows from a straightforward induction, observing that the second summand in the recursion is a morphism contained in $\cC_{n-2}$.
    \end{proof}
\end{lemma}
\begin{proposition}\label[proposition]{lemma:test}
The JW projectors $f_n^{ij}$ are non-zero idempotents and are annihilated when a cup or cap is added.
\begin{proof}
In this proof we will just write $U_n$ rather than $U_n(\omega_{ij},\omega_{ji})$ for simplicity. We will prove this lemma for the case when $n$ is odd and for adding a cap above the diagram, but the proof for adding a cup below and for even $n$ is very similar.
By expanding $f_{n-1}^{ij}$ via the recursion of \Cref{def:JW-proj}, we obtain the identity 
\begin{align}\label{eq:expanding1}
\resizebox{!}{3\baselineskip}{$\vcenter{\hbox{\MatchingBoxCapSA}}$}&\;=\;\resizebox{!}{3\baselineskip}{$\vcenter{\hbox{\MatchingBoxCapSB}}$}\,-\frac{U_{n-3}}{U_{n-2}}\, \resizebox{!}{3\baselineskip}{$\vcenter{\hbox{\MatchingBoxCapSC}}$} \\ \nonumber
&=\quad \left(\omega_{ij}-\frac{U_{n-3}}{U_{n-2}}\right) f_{n-2}^{ij}\quad 
= \quad \frac{U_{n-1}}{U_{n-2}}\,\, f_{n-2}^{ij}.
\end{align}

We assume (for induction) that the annihilating property holds and $(f_m^{ij})^2=f_m^{ij}$ for all $1\leq m<n$. These properties are clear for $n=0,1$, where there is nothing to check for the annihilating property. If we add a cap at any point other than over the $(n-1)^{th}$ and $n^{th}$ strands then there would be a cap over $f_{n-1}^{ij}$ which would annihilate to zero by the induction hypothesis. If we add a cap at this final two strands, then we get the left hand side of
\begin{align*}\label{eq:fin-1}
\resizebox{!}{2.5\baselineskip}{$\vcenter{\hbox{\MatchingBoxLOUCapT} }$} 
\!-\; \frac{U_{n-2}}{U_{n-1}}
\resizebox{!}{2.5\baselineskip}{$\vcenter{\hbox{\MatchingBoxLOUCapL} }$} 
=\; \resizebox{!}{2.5\baselineskip}{$\vcenter{\hbox{\MatchingBoxLOUCapFT} }$} -\; \resizebox{!}{2.5\baselineskip}{$\vcenter{\hbox{\MatchingBoxLOUCapFL} }$}=0.
\end{align*}
where in the first equality we substitute \eqref{eq:expanding1} and in the second equality we observe that all terms of $f_{n-2}^{ij}$ are annihilated by $f_{n-1}^{ij}$ in the composition besides the leading term $1$ so that the two terms cancel. 
Thus, the annihilating property holds for $f_n^{ij}$.

Now we use the annihilating property to inductively prove $(f_n^{ij})^2=f_n^{ij}$. In the case when $n$ is odd, we get
\begin{equation*}
(f_n^{ij})^2
=\resizebox{!}{2.5\baselineskip}{$\vcenter{\hbox{\MatchingBoxLOUA} }$}  -2\,\frac{U_{n-2}}{U_{n-1}}
\resizebox{!}{3\baselineskip}{$\vcenter{\hbox{\MatchingBoxLOUB} }$} 
+\left(\frac{U_{n-2}}{U_{n-1}}\right)^2
\resizebox{!}{3.5\baselineskip}{$\vcenter{\hbox{\MatchingBoxLOUC} }$} 
= f_{n}^{ij}.
\end{equation*}
Here, the second equality is obtained by simplifying the final summand in the expansion of $(f_n^{ij})^2$. The central part of this final summand equals $(U_{n-1}/U_{n-2})f_{n-2}^{ij}$ by \eqref{eq:expanding1}. We know that $(f_{n-2}^{ij}\otimes 1_{ji})f_{n-1}^{ij}=f_{n-1}^{ij}$ since $f_{n-1}^{ij}$ annihilates all terms of $f_{n-2}^{ij}\otimes 1_{ji}$ besides the leading term $1$. Therefore this final summand equals $(U_{n-2}/U_{n-1})f_{n-1}^{ij} h_{n-1}^{(i,j,\dots)} f_{n-1}^{ij}$, and substituting this into the formula for $(f_{n}^{ij})^2$ yields the second equality.
\end{proof}
\end{proposition}
We also observe the following uniqueness property of JW projectors.
\begin{lemma}\label[lemma]{lem: JW is unique}
Let $\bi=(i,j,\ldots)$ be an alternating sequence of length $n+1$. 
    If $f$ and $f'$ are non-zero idempotents of $[\bi]$  which are not contained in $\cC_{n-1}$ and annihilated by any cap or cup then $f=f'$.
    \begin{proof}
As $f$ and $f'$ are not contained in $\cC_{n-1}$, we can write $f = 1_\bi + g$ and $f' = 1_\bi + g'$, where $g$ and $g'$ consist only of diagrams in $\cC_{n-1}$. The coefficient of $1_\bi$ is necessarily equal to $1$ as diagrams in $\cC_{n-1}$ are closed under multiplication. Now, $g$ and $g'$ consist only of diagrams with caps and cups so they annihilate $f$ and $f'$. Then $f = f(1_\bi + g') = ff' = (1_\bi + g)f' = f'$.
    \end{proof}
\end{lemma}

Recall the concatenation $\bi\bj$ of two colour sequences, where the last colour of $\bi$ matches the first colour of $\bj$, from \Cref{def:concat}.

\begin{lemma}\label[lemma]{lem:alternating-sequence-dec}
  Let $\bi=(i_1, \ldots, i_{n+1})\in \{1,\ldots, \ell\}^{n+1}$ be a sequence of colours. There is a unique decomposition of $\bi=\bi_1\ldots \bi_k$ as a concatenation of alternating sequences $\bi_1,\ldots, \bi_k$ of maximal length. 
\end{lemma}
\begin{proof}
  To produce such a decomposition, find the first $k$ such that the $(k-1)$-th colour is not equal to the $(k+1)$-th colour and regard $k$ as the end of the first alternating sequence $\bi_1$. Now, continue with $(i_{k}, i_{k+1},i_{k+2}, \ldots)$ and find the next colour that breaks the alternating pattern and to identify $\bi_2$. Inductively, we can decompose any sequence this way.     
\end{proof}

\begin{example}\label[example]{ex:seq-dec}
The sequence $\bi = (1,1,1,2,1,2,1,2,2)$ decomposes into the alternating sequences $\bi_1=(1,1,1)$, $\bi_2=(1,2,1,2,1,2)$, and $\bi_3=(2,2)$.
\end{example}

We denote the coloured JW projector $f_n^{ij}$ by $f_{\bi}$, where $\bi=(i,j,i,\ldots)$ is the alternating sequence of length $n+1$. For example, $f_{(1,2,1,2)} = f_3^{12}$. We will now generalise this construction to arbitrary sequences of colours.

\begin{definition}\label[definition]{def:f-bi}
Let $\bi=(i_1, \ldots, i_{n+1})\in \{1,\ldots, \ell\}^{n+1}$ be a sequence of colours. The \emph{coloured JW projector} associated to $\bi$ is defined as the tensor product
$$f_{\bi}:=f_{n_1}^{i_1,j_1}\otimes \ldots \otimes f_{n_k}^{i_k,j_k},$$
for the decomposition $\bi=\bi_1\ldots \bi_k$ of $\bi$ as a concatenation of alternating sequences
 $\bi_a=(i_a,j_a,\ldots)$ of length $n_{a}+1$, for $a=1,\ldots, k$,  from \Cref{lem:alternating-sequence-dec}
\end{definition}

\begin{example}
With the sequence $\bi=(1,1,1,2,1,2,1,2,2)$ from \Cref{ex:seq-dec}, we have
\begin{equation*}
f_\bi = f_2^{1,1}\otimes f_5^{1,2}\otimes f_1^{2,2}.
\end{equation*}
\end{example}

\begin{proposition}\label[proposition]{prop: new jws work}
    The coloured JW projectors $f_\bi$ are the unique non-zero idempotents whose coefficient of $1_\bi$ is equal to $1$ that are annihilated by adding any caps or cups.
    \begin{proof}
    By fuctoriality of the tensor product, $(f\otimes g)(f'\otimes g)=ff'\otimes gg'$, we have that $f_\bi$ is again an idempotent.
    Moreover, the coefficient of $1_\bi$ in $f_\bi$ must also be $1$ due to compatibility of tensor product and identities. Hence, $f_\bi\neq 0$.

Write $f_{\bi}=f_{\bi_1}\otimes \ldots \otimes f_{\bi_k}$ for a decomposition of $\bi$ into alternating sequences $\bi_a$ as in \Cref{lem:alternating-sequence-dec}. 
Any caps or cups added on neighbouring strands that are part of an alternating sequence $\bi_a$ will be annihilated by $f_{\bi_a}$. Now consider adding a cap at the $(t-1)$-th and $t$-th strand at the break of two neighbouring alternating sequences. This means that the colours $i_{t-1}$ and $i_{t+1}$ are distinct. This leads to the multiplication of the disjoint idempotents $e_{i_{t-1}}$ and $e_{i_{t+1}}$ in the region connected by this cap or cup in every diagram that appears in $f_{\bi}$, hence, giving zero. 
    \end{proof}
\end{proposition}
\begin{remark}
    The construction of coloured JW projectors for the subcategory $S\mathcal{TL}(\omega)$ from \Cref{rem:EL-version} was already given in \cite{EL}*{Section 2.6}. In the two-coloured case, this goes back to \cite{Eli}*{Appendix A.6} and was generalised in \cite{Haz}*{Theorem~A} to any commutative base ring.
\end{remark}

\subsection{Semisimplicity}

We now prove the first main result of this article regarding semisimplicity of $\TL(\KK^\ell, \omega)$.

\begin{definition}\label[definition]{def:T-bi}
If for a sequence of colours $\bi$ of length $n+1$ the JW projector $f_\bi$ of \Cref{def:f-bi} exists, then we define $T_\bi:= ([\bi], f_\bi)$ to be the corresponding object in $\TL(\KK^\ell, \omega)$.
\end{definition}

\begin{theorem}\label{thm: TL is semisimple}
    The category $\cC = \TL(\KK^\ell,\omega)$ is semisimple if and only if \Cref{assumption:chebychev-non-zero} is satisfied.
\end{theorem}
    \begin{proof}

First assume that \Cref{assumption:chebychev-non-zero} is satisfied. We want to show that $\cC$ is semisimple. 
For $\bi$ a colour sequence of length $n+1$, we first show that $\End_{\cC}(T_\bi) = \KK$. 
 Since   $\End_{\cC}(T_\bi) = f_\bi \End_{\cC}([\bi])f_\bi$, and $\End_\cC([\bi])$ is the $\KK$-span of $f_\bi gf_\bi $ for all diagrams $g$ of $n+1$ stands with colours $\bi$ at the top and bottom, if $g$ has any caps or cups then it will annihilate $f_\bi$. The only non-zero option is then $g =1_\bi$, and so $\End_{\cC}(T_\bi)$ is one dimensional and $T_\bi$ is an indecomposable object. 
 
 Moreover, for any two such sequences $\bi \neq \bj$ we have $\Hom_\cC(T_\bi, T_\bj) = 0$. This follows as only the identity diagram would not be annihilated by composing with the JW projectors, but in this diagram with colours $\bi$ at the top and $\bj$ at the bottom orthogonal idempotents are multiplied, giving zero.
    
Now, we consider the image $\pi(T_\bi)$ in $\cC_n/\cC_{n-1}$. The coefficient of $1_\bi$ in $f_\bi$  is $1$, and all the other diagrams in $f_\bi$ will factor through an object of $\cC_{n-1}$. Therefore every diagram in $1_\bi-f_\bi$ factors through an object of $\cC_{n-1}$, so $\pi(1_{([\bi], 1_\bi-f_\bi)}) = \pi(1_\bi-f_\bi)) = 0$ and $\pi(([\bi], 1_\bi-f_\bi))\cong 0$. Therefore $\pi([\bi]) = \pi(T_\bi)$. As $T_\bi$ is indecomposable in $\cC_n$ and $T_\bi\notin \cC_{n-1}$,  we must have $X_\bi = T_\bi$. By \Cref{prop:bijection of indecomposables}, every indecomposable object in $\cC_n$ is isomorphic to one of the $T_\bi$. This shows that every indecomposable object in $\cC$, $X_\bi = T_\bi$, is simple and since $\cC$ is KS this implies that every object in $\cC$ is isomorphic to a finite direct sum of simple objects. Therefore $\cC = \TL(\KK^\ell,\omega)$ is semisimple by \Cref{lem: object semisimple implies semisimple} and the objects $T_\bi$ are simple objects.

Conversely, assume that \Cref{assumption:chebychev-non-zero} is not satisfied. We will construct an indecomposable object that is not simple. 
There exist $1\leq i,j \leq \ell$ such that $U_n(\omega_{ij}, \omega_{ji}) = 0$ for some $n\geq 0$. We can always choose $n$ to be minimal, so that $U_k(\omega_{ij}, \omega_{ji})\neq 0$ for all $1\leq k\leq n$. This means that for the alternating sequence $(i,j,i,j,\dots)$ the JW projectors $f_k$ will exist for all $1\leq k \leq n$, and we can find them using the recursion formula in \Cref{def:JW-proj}. We will therefore focus on the sequence $\bi = (i,j,i,\dots)$ of length $n+2$, for which there is no JW projector defined by the recursion.

For this sequence we consider the idempotent
\begin{equation}f = \begin{cases}
    f_n^{ij}\otimes 1_{(ij)}, &\text{if }n\text{ is even,}\\
    f_n^{ij}\otimes 1_{(ji)}, &\text{if }n\text{ is odd.}
\end{cases}\end{equation} This $f$ provides a splitting of $[\bi]$ into a direct sum $[\bi] = ([\bi], f) \oplus ([\bi], 1_\bi-f)$, and since $\pi([\bi])$ is simple in $\cC_{n+1}/C_n$ one of either $\pi(([\bi], f))$ or $\pi(([\bi], 1_\bi-f))$ must be $0$. We know by \Cref{lem: coeff of identity diagram is 1} that the coefficient of the identity diagram in $f_n^{ij}$ is $1$, so the coefficient of $1_\bi$ in $1_\bi - f$ is $0$. Therefore every diagram in $1_\bi - f$ will factor through an object in $\cC_n$, so $\pi(([\bi], 1_\bi-f))\cong 0$ and $\pi(([\bi], f))\cong \pi([\bi])$.

We claim that the object $([\bi], f)$ described above is indecomposable but not simple.
The endomorphisms of $([\bi], f)$ in $\cC$ are $f \End_\cC([\bi])f$, which is the $\KK$-span of $fgf$ for all diagrams $g$ of $n+1$ stands with colours $\bi$ at the top and bottom. The definition of $f$ implies that if $g$ has any caps or cups within the first $n$ strands then $fgf = 0$. This means the first $n-1$ strands must just be vertical lines, and the only options for $g$ are $1_\bi$ or $h_n^\bi$ and $\End_\cC(([\bi],f))$ is the $\KK$-span of $f1_\bi f$ and $e_\bi:=fh_n^\bi f$. 
Graphically, we have
        \begin{equation*}
        e_\bi \; = \vcenter{\hbox{\resizebox{!}{6\baselineskip}{\GeneratoreEven}}}
        \end{equation*}
        When $n$ is even, using the same argument as in \eqref{eq:expanding1}, we find that
        \begin{align*}        e_\bi^2=\vcenter{\hbox{\resizebox{!}{5\baselineskip}{\ZeroSquareEven}}} 
= \frac{U_{n}(\omega_{ij},\omega_{ji})}{U_{n-1}(\omega_{ij},\omega_{ji})}\vcenter{\hbox{\resizebox{!}{5.5\baselineskip}{\ZeroSquareEvenTwo}}},
        \end{align*} 
        which is equal to zero since $U_n(\omega_{ij},\omega_{ji})=0$ by assumption. Similarly if $n$ is odd we get $e_\bi^2 = \left(\omega_{ji} - \frac{U_{n-2}(\omega_{ij},\omega_{ji})}{U_{n-1}(\omega_{ij},\omega_{ji})}\right)f_{n-1}^{ij} = \frac{U_{n}(\omega_{ij},\omega_{ji})}{U_{n-1}(\omega_{ij},\omega_{ji})}f_{n-1}^{ij} = 0$. In summary, we have shown that $\End_\cC(([\bi],f))$ is the $\KK$-span of $f$ and $e_\bi$, $f^2 = f$ and $e_\bi^2 = 0$. Together with the fact that $fe_\bi = e_\bi f = e_\bi$ this shows that $\End_\cC(([\bi],f))\cong \KK[e_\bi]/(e_\bi^2)$. This is a local ring so $([\bi],f)$ is an indecomposable object in $\cC$. Its Jacobson radical is $(e_\bi)\neq 0$ so $\End_\cC(([\bi],f))$ is not semisimple and hence $\cC = \TL(\KK^\ell,\omega)$ is not semisimple.
    \end{proof}

Note that when $\ell=1$, so that $\TL(\KK,\omega)=\TL(d)$, for some $d\in \KK$, \Cref{thm: TL is semisimple} specialises to the statement that $\TL(d)$ is semisimple if and only if $d$ is not a root of a Chebychev polynomial. Equivalently, $d\neq q+q^{-1}$ for a root of unity $q$. This characterisation of semisimplicity of classical TL categories is well-known, see e.g. \cite{DG}, and can be derived from the existence of JW projectors in \cite{Wen}. 

\subsection{Gram determinants and meanders}
\label{sec:Gram-meander}

We fix the trace $\tr\colon \KK^\ell\to \KK$, $e_i\mapsto 1$. This trace induces a non-degenerate pairing 
$$(-,-)\colon \KK^\ell \times \KK^\ell \to \KK, \quad (a,b):=\tr(ab),$$
making $\KK^\ell$ a Frobenius algebra. 

We can extend this trace to a trace $\tr\colon \TL_n(\KK^\ell,\omega)\to \KK$  by closing the strings of a diagram, i.e., by $\KK$-linearly extending the assignment 
\begin{equation}
c\;\longmapsto \;\tr\Bigg(\vcenter{\hbox{\resizebox{!}{4\baselineskip}{\TraceTL}}}\Bigg),
\end{equation}
where $c$ is a crossingless matching with coloured regions. This gives a pairing 
\begin{equation}\label{eq:trace-pairing-TL}
    (-,-)\colon  \TL_n(\KK^\ell,\omega)\times  \TL_n(\KK^\ell,\omega)\to \KK, \quad (a,b):=\tr(ab^t),
\end{equation}
where $(-)^t\colon \TL_n(\KK^\ell,\omega)\to \TL_n(\KK^\ell,\omega)$ denotes the \emph{transposition} defined by $\KK$-linearly extending reflection along the horizontal axis of crossingless matchings. The pairings \eqref{eq:trace-pairing-TL} are bilinear pairings on the endomorphism algebras $\End_\cC([n])=\TL_n(\KK^\ell,\omega)$. They restrict to pairings on $\End_\cC(X)$ for any direct summand $X$ of $[n]$. If \Cref{assumption:chebychev-non-zero} holds, we have simple objects $T_\bi$ for each sequence $\bi$ of colours. The restriction of the pairing to $\End_\cC(T_\bi)$ is determined by
$$(f_\bi,f_\bi)=\tr(f^2_\bi)=\tr(f_\bi)=\tr(\id_{T_\bi}),$$
the \emph{quantum dimension} of $T_\bi$.

\begin{lemma}\label[lemma]{lem:q-dim}
    The quantum dimension of the simple object $T_\bi$ are products of Chebychev polynomials $U_n(\omega_{ij},\omega_{ji})$, with $i,j\in \{1,\ldots, \ell\}$.
\end{lemma}
\begin{proof}
    We start with an alternating sequence $\bi=(i,j,\ldots)$ of length $n+1$ and denote $f_n=f_n^{ij}$. Clearly, $\tr(f_0)=\tr(e_i)=1$ and $\tr(f_1)=\omega_{ji}$. Moreover, we derive the recursion 
    $$\tr(f_n)=\begin{cases}
   \big( \omega_{ji} - \frac{U_{n-2}(\omega_{ij}, \omega_{ji})}{U_{n-1}(\omega_{ij}, \omega_{ji})}\big) \tr(f_{n-1}) & \text{($n$ odd)},\\
   \big( \omega_{ij} - \frac{U_{n-2}(\omega_{ij}, \omega_{ji})}{U_{n-1}(\omega_{ij}, \omega_{ji})}\big) \tr(f_{n-1}) & \text{($n$ even)}
    \end{cases} \Bigg\rbrace= \frac{U_{n}(\omega_{ij}, \omega_{ji})}{U_{n-1}(\omega_{ij}, \omega_{ji})}  \tr(f_{n-1}),
    $$
    from \Cref{def:JW-proj} and \Cref{def:two-var-Cheb}. By induction, $\tr(f_n)=U_{n}(\omega_{ij}, \omega_{ji}).$
    The JW projectors for a general sequence $\bi$ are tensor products of the ones of alternating sequences appearing in the decomposition from \Cref{lem:alternating-sequence-dec}. Hence, the claim follows from $\tr(f\otimes g)=\tr(f)\tr(g)$. 
\end{proof}

We can now prove a generalisation of \cite{KL}*{Conjecture 23} on the Gram matrices, the matrices representing the bilinear pairing \eqref{eq:trace-pairing-TL} on $\TL_n(\KK^\ell, \omega)$.

\begin{corollary}\label[corollary]{cor:conjecture23}
Up to sign, the Gram determinants of the pairings \eqref{eq:trace-pairing-TL} are products of Chebychev polynomials $U_n(\omega_{ij},\omega_{ji})$, with $i,j\in \{1,\ldots, \ell\}$. In particular, the pairings are non-degenerate for all $n$ if and only if \Cref{assumption:chebychev-non-zero} holds. 
\end{corollary}
\begin{proof}
Assume that for some $n\geq 0$, $U_k(\omega_{ij,ji})\neq 0$ for all $1\leq k< n$. Then the JW projectors $f_{\bi}$ from \Cref{def:f-bi} exist for all sequences $\bi$ of colours of length less than or equal to $n+1$. Hence, $[n]$ is a direct sum of the corresponding simple objects $T_\bi$.  By choosing a suitable basis for $TL_n(\KK^\ell,\omega)$ consisting of identities on these objects (or isomorphisms between different copies of the same simple object), we can achieve that the Gram matrix has a sparse form, with only one non-zero entry in each row and each column. The non-zero entries are of the form $\tr(f_n)$. Hence, by \Cref{lem:q-dim}, the Gram determinant is a product of two-variable Chebychev polynomials $U_k(\omega_{ij},\omega_{ji})$, $1\leq k<n$, up to signs from permuting rows and columns. 

Conversely, if $U_n(\omega_{ij},\omega_{ji})=0$ for some $n$, then the pairing \eqref{eq:trace-pairing-TL} for $n+1$ is degenerate. Indeed, decomposing $[n+1]$ as a direct sum of indecomposable objects, we can restrict the pairing to $\End_\cC(X_{\bi})$, where $X_\bi$ denotes the indecomposable object associated to $\bi=(i,j,\ldots)$, the corresponding alternating sequence of length $n+2$. As $X_\bi$ is not simple, as shown in the proof of \Cref{thm: TL is semisimple}, the pairing is degenerate as it contains the radical of $\End_\cC(T_{\bi})$. This forces the Gram determinant of the pairing for the  object $[n+1]$ to be zero.
\end{proof}

\begin{remark}
In \cite{KL}*{Definition~11}, the concept of a \emph{$\KK$-spherical} wrapping map $\omega$ is defined. In the semisimple case, choosing a non-degenerate pairing $\tr\colon \KK^\ell \to \KK, e_i\mapsto b_i$, this is equivalent to the matrix $\omega=(\omega_{ij})$ being symmetrisable and if $ \KK$ is algebraically closed, we can choose $c_{ij}=\sqrt{b_i/b_j}\omega_{ij}$ such that $c_{ij}=c_{ji}$. In this case, 
$U_n(c_{ij},c_{ji})=U_n(c_{ij})$ is just the ordinary Chebychev polynomial of the second kind. Then \Cref{cor:conjecture23} specifies to the statement that the Gram determinants are products of Chebychev polynomials of the second kind in $c_{ij}$ and degeneracy occurs if and only if there exists a pair $i,j$ such that $c_{ij}=q+q^{-1}$ for a root of unity $q$. This is the statement of \cite{KL}*{Conjecture~23}.
\end{remark}

For classical TL algebras, the Gram determinant for the pairings was computed in \cites{DFGG,DiFrancesco} and referred to as the \emph{Meander determinant}. A \emph{meander} is a configuration of a closed non-self-intersecting loop crossing an infinite line, see the left diagram in \Cref{fig:meanders}. There are $2n$ points where the crossing occurs.
\begin{figure}
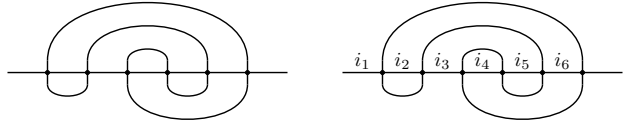

    \centering
    $\resizebox{!}{2.5\baselineskip}{$\vcenter{\hbox{\Meander}}$} \qquad  \resizebox{!}{2.5\baselineskip}{$\vcenter{\hbox{\MeanderLabelled}}$}$
    \caption{A meander on the left, and the same meander with coloured regions on the right.}
    \label{fig:meanders}
\end{figure}
Each meander can be obtained by pairing the diagram contained in the upper and lower half of the infinite line. Thus, the trace 
$\tr\colon \TL_n(d)\to \KK$, in particular, evaluates each meander diagram $m$ to $d^{|m|}$, where $|m|$ is the number of connected components.

The more general framework of crossingless matchings with labelled regions studied in this article allows to consider meanders with coloured regions as illustrated in the right diagram of \Cref{fig:meanders}.  
The results of this section can be interpreted as computing a coloured version of the meander determinant. 

\section{Tensor product decompositions}
\label{sec:tensorproduct}

In this section we assume that the two-variable Chebychev polynomials do not vanish on pairs $(\omega_{ij},\omega_{ji})$, as in \Cref{assumption:chebychev-non-zero}, so $\cC=\TL(\KK^\ell,\omega)$ is semisimple and every colour sequence $\bi$ has an associated simple object $T_\bi$. In addition, assume that the field $\KK$ is algebraically closed.

 To compute tensor product decompositions for the simple objects $T_\bi$ from \Cref{def:T-bi}, assume that $\bi=(i_1,\ldots, i_{n+1})$ and $\bj=(j_1,\ldots, j_{m+1})$ are sequences of indexes. For the tensor product $T_\bi\otimes T_\bj$ to be non-zero we require $i_{n+1}=j_1$. If, moreover, $i_n\neq j_2$ then any cap or cup will annihilate this tensor product and so $f_\bi \otimes f_\bj  = f_{\bi\bj}$ by the uniqueness of the JW projector \Cref{lem: JW is unique}, where $\bi\bj$ is the concatenation from \Cref{def:concat}. Therefore in this case $T_\bi \otimes T_\bj = T_{\bi\bj}$. We now consider the case when $\bi$ and $\bj$ have more then just one colour in common.
\begin{proposition}
    For colour sequences $\bi$ and $\bj$, let $r$ be the number of colours at the start of $\bj$ that match with colours at the end of $\bi$, so $j_1 = i_{n+1}$, $j_2 = i_n$, $\dots$, $j_r = i_{n+2-r}$ and $j_{r+1}\neq i_{n+1-r}$. Then $\End_\cC(T_\bi \otimes T_\bj)$ is an $r$-dimensional algebra with basis 
    \begin{equation}
    b_k = \resizebox{0.3\textwidth}{!}{$\vcenter{\hbox{\BkBasis}}$}, \qquad \text{for } k = 0, \dots, r-1.
    \end{equation}
    \begin{proof}
For any diagram $g$ of crossingless matchings with colour sequence $\bi\bj$ at the top and bottom, we consider the restriction 
$g'=(f_\bi\otimes f_\bj)g(f_\bi\otimes f_\bj)\in \End_\cC(T_\bi \otimes T_\bj).$
Using the annihilating property for $f_\bi$ and $f_\bj$, the only possible non-vanishing elements $g'$ are those where $g$ is one of the following diagrams, with the arcs going between $f_\bi$ and $f_\bj$:
        \begin{equation}\label{diagram: central arc diagrams}
              g_0 = \vcenter{\hbox{\CentralArcsA}},\; g_1 = \vcenter{\hbox{\CentralArcsB}},\; g_2 = \vcenter{\hbox{\CentralArcsC}}, \ldots .
        \end{equation}
Any other crossingless matching diagram $g$ connects at least two strands in $\bi$ or in $\bj$ at the top or the bottom and hence its restriction $g'$ is zero. The corresponding restrictions $g_i'$, where $g_i$ is any of the diagrams from \eqref{diagram: central arc diagrams}, span $\End_\cC(T_\bi \otimes T_\bj)$. By assumption $j_{r+1}\neq i_{n+1-r}$, while $j_k=i_{n+2-k}$, for all $k=1,\ldots, r$, so only the restrictions $g'_0,\ldots, g'_{r-1}$ remain nonzero and $g_k'=b_k$ for all $k=0,\ldots, r-1$.

These elements $b_0,\ldots, b_{r-1}$ are linearly independent as they are contained in different parts of the filtration, with $b_0$ being in $\cC_{n+m}\setminus \cC_{n+m-1}$, $b_1$ in $\cC_{n+m-2}\setminus \cC_{n+m-3}$, etc, and $b_{r-1}$ in $\cC_{n+m-2r}\setminus \cC_{n+m-2r-1}$. 
    \end{proof}
\end{proposition}

\begin{proposition}\label[proposition]{prop:basis2}
    The following elements of $\End_\cC(T_\bi \otimes T_\bj)$ also form a basis.
    \begin{equation}
    a_k = \resizebox{0.35\textwidth}{!}{$\vcenter{\hbox{\AkBasis}}$}, \quad \text{for }k = 0,\dots, r-1.
    \end{equation}
    where $\bi \cap^k\bj = (i_1, \dots, i_{n+1-k}, j_{k+2}, \dots, j_{m+1}) = (i_1, \dots, i_{n-k}, j_{k+1}, \dots, j_{m+1})$ is the subsequence of $\bi\bj$ with the relevant overlapping parts removed. 
    
    This basis consists of pairwise orthogonal quasi-idempotents in the sense that $a_ka_l=0$ for $k\neq l$ and $a_k^2=\lambda a_k$, for $\lambda_k\in \KK\setminus\{0\}$. In particular, $\End_\cC(T_\bi \otimes T_\bj)$ is a commutative algebra.
\end{proposition}
    \begin{proof}
          To show that the $a_k$ form a basis, we will show that each of the $b_k$ can be written as a linear combination of the $a_k$. Therefore the $a_k$ must span $\End_\cC(T_\bi \otimes T_\bj)$, and there are $r$ of them so they must be a basis.

We expand  $f_{\bi \cap^{r-1} \bj}$ as $1_{\bi \cap^{r-1} \bj}$ plus terms of lower order in the filtration. We know that $a_k\in \cC_{n-2k}\setminus \cC_{n-2k-1}$. Hence, $a_k=b_k + \lambda_{k,1}b_{k+1}+\ldots$ for some $\lambda_{k,i}\in \KK$. The coefficients of any $b_l$ with $l<k$ are zero and the leading coefficient of $b_k$ equals $1$ in this expansion.  Hence, the $a_k$ also form a basis.     

The $a_k$ are also orthogonal with respect to multiplication. If we take $a_sa_t$ for some $s<t$, we see that $a_t$ factors through $[n+m-2t]$ and $f_{\bi\cap^s\bj}$ is a JW projector with $n+m-2s$ strands. Since $s<t$ this implies that $a_sa_t=0$. This works similarly for $s>t$. In particular, each basis element commutes with all other basis elements and  $\End_{\cC}(T_\bi)$ is a commutative algebra.

We also see that $a_k^2 = \sum_{i=0}^{r-1}\alpha_i a_i$ for some $\alpha_i\in\KK$, and for any $s\neq k$ we have $0 = a_s a_k^2 = \alpha_s a_s^2$. By \Cref{lem: no central nilpotent in semisimple algebra} we must have $a_s^2 \neq 0$, using that $\End_\cC(T_\bi \otimes T_\bj)$ is semisimple. Therefore we must have $\alpha_s = 0$ for all $s\neq k$, and so $a_k^2 = \lambda_k a_k$. Again $a_k^2\neq 0$ so $\lambda_k \neq 0$.
\end{proof}

We will normalise the $a_k$ to form a basis of orthogonal idempotents in the following. This is possible since $\KK$ is algebraically closed. We can now give a tensor product decomposition of $T_\bi \otimes T_\bj$.

\begin{theorem}\label{thm:tensor-dec}
    For colour sequences $\bi = (i_1, \dots, i_{n+1})$ and $\bj = (j_1, \dots, j_{m+1})$, let $r$ be the number of colours at the start of $\bj$ that match with colours at the end of $\bi$, so $j_1 = i_{n+1}$, $j_2 = i_n$, $\dots$, $j_r = i_{n+2-r}$ and $j_{r+1}\neq i_{n+1-r}$. If $r=0$ then $T_\bi\otimes T_\bj \cong 0$, and if $r>0$, then $\End_\cC(T_\bi\otimes T_\bj) \cong \KK^{r}$ and 
    \begin{equation}\label{eq:tensor-decomp}
    T_\bi\otimes T_\bj = \bigoplus_{k = 0}^{r-1} T_{\bi \cap^k \bj}.
    \end{equation}
    \begin{proof}
        If $r=0$ this means $j_1\neq i_{n+1}$, but these colours will meet and annihilate in $1_\bi \otimes 1_\bj = 1_{T_\bi\otimes T_\bj}$ so $1_{T_\bi\otimes T_\bj}=0$ and $T_\bi\otimes T_\bj \cong 0$.

        If $r>0$, each $a_k$ is an idempotent of $T_\bi \otimes T_\bj$ and factors through the simple object $T_{\bi \cap^k \bj}$, so this must be a direct summand $T_{\bi \cap^k \bj}\leq_\oplus T_\bi \otimes T_\bj$.  Consider the element $x = \sum_{k = 0}^{r-1} a_k \in \End_\cC(T_\bi \otimes T_\bj)$. For any basis idempotent $a_k$ we have $xa_k = a_kx = a_k$, and since the unit of an algebra must be unique we have $x = 1_{\End_\cC(T_\bi \otimes T_\bj)} = 1_{T_\bi \otimes T_\bj}$. This proves \eqref{eq:tensor-decomp}.
        
        If $s\neq k$ then $f_{\bi\cap^s\bj}$ and $f_{\bi\cap^k\bj}$ will have a different number of strands so any morphism between $T_{\bi\cap^s\bj}$ and $T_{\bi\cap^k\bj}$ must contain caps or cups and will therefore annihilate. This means $\Hom_\cC(T_{\bi\cap^s\bj}, T_{\bi\cap^k\bj}) = 0$ and, using the fact that each $T_{\bi\cap^k \bj}$ is simple, we see that $$\End_\cC(T_\bi\otimes T_\bj) \cong \prod_{k=0}^{r-1}\End_\cC(T_{\bi\cap^k \bj}) = \KK^{r}.$$
        This completes the proof of the theorem.
        \end{proof}
\end{theorem}

We observe that in the direct sum decomposition of the tensor product in \eqref{eq:tensor-decomp}, each simple summand will have multiplicity no more than $1$.
Special cases of the tensor product decompositions were already given in \cite{EL}*{Corollary 2.18}.

Under \Cref{assumption:chebychev-non-zero}, $\cC=\TL(\KK^\ell,\omega)$ is a semisimple \emph{multitensor category} in the terminology of \cite{EGNO} as it is a hom-finite semisimple monoidal category which is rigid, i.e., has left and right duals. The tensor unit decomposition from \eqref{eq:unit-dec} can be used to decompose 
$$\cC=\textstyle \bigoplus_{i,j=1}^\ell \cC_{ij},$$
where $\cC_{ij}$ is the full subcategory consisting of direct sums of objects $T_{\bi}$ where the colour sequence $\bi$ starts with the colour $i$ and ends with the colour $j$. The $\cC_{ii}$ are tensor categories, i.e., have a simple tensor unit $\one_i=[(i)]$.

\begin{example}
    If $\ell=1$, denote by $T_{n}$ the simple object corresponding to the sequence $(1,\ldots, 1)$ of length $n+1$. Then $T_n\otimes T_m=\bigoplus_{r=0}^{\lfloor n+m \rfloor} T_{n+m-2r}$ \Cref{thm:tensor-dec}, which recovers the well-known tensor product decomposition in the semisimple case of the classical TL category, also known as the Clebsch--Gordan rule.
\end{example}

\appendix

\section{Chebychev Polynomials}\label[appendix]{app: chebychev}

A key tool used in \Cref{sec:semisimple} are the two-variable Chebychev polynomials $U_n(x,y)$ defined in \Cref{def:two-var-Cheb}. These polynomials were already studied in \cite{Haz}*{Section~3}, where $[n+1]_s = U_n(x_t, x_s)$. In the following, we prove some properties of these two-variable Chebychev polynomials. These are not used in the main text but might be of independent interest. 

For simplicity we just write $U_n=U_n(x,y)$, and $\overline{U_n}=U_n(y,x)$.

\begin{lemma}\label[lemma]{lem: Chebychev expanded formula}
    Let $1<p<m$, then $$U_m = \begin{cases}
        \overline{U_p}U_{m-p} - \overline{U_{p-1}}U_{m-p-1}, &\text{if $m$ is even}, \\
        {U_p}U_{m-p} - {U_{p-1}}U_{m-p-1}, &\text{if $m$ is odd}.
    \end{cases}$$
    \begin{proof}
    When $p=1$, these formulas become the recursion formulas that define the Chebychev polynomials. For $p>1$, we use induction on $p$. Here we will give the proof for the case where $m$ is even and $p+1$ is odd, but the other cases use the same method. Assume for induction that $U_m = \overline{U_p}U_{m-p} - \overline{U_{p-1}}U_{m-p-1}$ for some even $p$. Then using that $m-p$ is even and that $p+1$ is odd, we get
    \begin{align*}
        U_m &= \overline{U_p}(xU_{m-p-1} - U_{m-p-2}) - \overline{U_{p-1}}U_{m-p-1}\\
        &= (x\overline{U_p} - \overline{U_{p-1}})U_{m-p-1} - \overline{U_p}U_{m-p-2}\\
        &=\overline{U_{p+1}}U_{m-(p+1)} - \overline{U_{(p+1)-1}}U_{m-(p+1)-1}
    \end{align*}
By induction, the statement follows. 
    \end{proof}
\end{lemma}
\begin{lemma}\label[lemma]{lem: chebychev overline formula}
    Let $m\geq 0$, then $$\overline{U_m} = \begin{cases}
        U_m, &\text{if $m$ is even}, \\
        \frac{x}{y}{U_m}, &\text{if $m$ is odd}.
    \end{cases}$$
    \begin{proof}
        First let $m$ be even, then we apply \Cref{lem: Chebychev expanded formula} with $p = m-1$ to get $U_m = \overline{U_{m-1}}U_{1} - \overline{U_{m-2}}U_{0} = y\overline{U_{m-1}} - \overline{U_{m-2}}$ which is the formula for $\overline{U_m}$ when $m$ is even. 
        
        Let $m$ be odd, so $\overline{U_m} = x\overline{U_{m-1}} - \overline{U_{m-2}}$, and assume for induction that $\overline{U_{m-2}} = \frac{x}{y}U_{m-2}$. Then $\overline{U_m} = x{U_{m-1}} - \frac{x}{y}{U_{m-2}} = \frac{x}{y}(yU_{m-1} - U_{m-2}) = \frac{x}{y}U_m$, and since this is clearly true for the first few $U_m$ must hold for all $m$.
    \end{proof}
\end{lemma}
The final proposition shows an equivalent condition for divisibility of the two-variable Chebychev polynomials. 

\begin{proposition}\label[proposition]{prop: chebychev divisibility}
    For $m,n \geq 0$, $U_m ~|~ U_n$ if and only if $m+1 ~|~ n+1$.
    \begin{proof}
    We know from \cite{Haz}*{Section 3} that there exist distinct irreducible polynomials $\Theta_{n,s} \in \ZZ[x,y]$ for $n\geq 1$, such that $U_n(x,y) = [n+1]_s = \prod_{k|n+1} \Theta_{k,s}$. Therefore if $U_m ~|~ U_n$ then  $\prod_{k|m+1} \Theta_{k,s}~|~\prod_{l|n+1} \Theta_{l,s}$, and in particular $\Theta_{m+1,s}~|~\prod_{l|n+1} \Theta_{l,s}$. Therefore $\Theta_{m+1,s}= \Theta_{l,s}$ for some $l|n+1$, so $m+1~|~n+1$.
    
    Conversely if $m+1~|~n+1$, then for any $k~|~m+1$ we have $k~|~n+1$ so $\Theta_{k,s}~|~\prod_{l|n+1} \Theta_{l,s}$. Then since all the $\Theta_{n,s}$ are distinct and irreducible we have $\prod_{k|m+1} \Theta_{k,s}~|~\prod_{l|n+1} \Theta_{l,s}$ and $U_m ~|~ U_n$.
    \end{proof}
\end{proposition}


\bibliography{biblio}

\begin{bibdiv}
\begin{biblist}

\bib{assem}{book}{
      author={Assem, Ibrahim},
      author={Simson, Daniel},
      author={Skowro\'{n}ski, Andrzej},
       title={Elements of the representation theory of associative algebras.
  {V}ol. 1},
      series={London Mathematical Society Student Texts},
   publisher={Cambridge University Press, Cambridge},
        date={2006},
      volume={65},
        ISBN={978-0-521-58423-4; 978-0-521-58631-3; 0-521-58631-3},
         url={https://doi.org/10.1017/CBO9780511614309},
        note={Techniques of representation theory},
}

\bib{DiFrancesco}{article}{
      author={Di~Francesco, P.},
       title={Meander determinants},
        date={1998},
        ISSN={0010-3616},
     journal={Comm. Math. Phys.},
      volume={191},
      number={3},
       pages={543\ndash 583},
         url={https://doi.org/10.1007/s002200050277},
}

\bib{DFGG}{article}{
      author={Di~Francesco, P.},
      author={Golinelli, O.},
      author={Guitter, E.},
       title={Meanders and the {T}emperley-{L}ieb algebra},
        date={1997},
        ISSN={0010-3616},
     journal={Comm. Math. Phys.},
      volume={186},
      number={1},
       pages={1\ndash 59},
         url={https://doi.org/10.1007/BF02885671},
}

\bib{DG}{article}{
      author={Doty, Stephen},
      author={Giaquinto, Anthony},
       title={Origins of the {T}emperley--{L}ieb algebra: early history},
        date={2023},
     journal={arXiv e-prints},
      eprint={ArXiv:2307.11929},
}

\bib{Eli}{article}{
      author={Elias, Ben},
       title={The two-color {S}oergel calculus},
        date={2016},
        ISSN={0010-437X},
     journal={Compos. Math.},
      volume={152},
      number={2},
       pages={327\ndash 398},
         url={https://doi.org/10.1112/S0010437X15007587},
}

\bib{EL}{article}{
      author={Elias, Ben},
      author={Libedinsky, Nicolas},
       title={Indecomposable {S}oergel bimodules for universal {C}oxeter
  groups},
        date={2017},
        ISSN={0002-9947},
     journal={Trans. Amer. Math. Soc.},
      volume={369},
      number={6},
       pages={3883\ndash 3910},
         url={https://doi.org/10.1090/tran/6754},
        note={With an appendix by Ben Webster},
}

\bib{EGNO}{book}{
      author={Etingof, Pavel},
      author={Gelaki, Shlomo},
      author={Nikshych, Dmitri},
      author={Ostrik, Victor},
       title={Tensor categories},
      series={Mathematical Surveys and Monographs},
   publisher={American Mathematical Society, Providence, RI},
        date={2015},
      volume={205},
        ISBN={978-1-4704-2024-6},
         url={https://doi.org/10.1090/surv/205},
}

\bib{FLP}{article}{
      author={Flake, Johannes},
      author={Laugwitz, Robert},
      author={Posur, Sebastian},
       title={Indecomposable objects in {K}hovanov-{S}azdanovic's
  generalizations of {D}eligne's interpolation categories},
        date={2023},
        ISSN={0001-8708},
     journal={Adv. Math.},
      volume={415},
       pages={108892, 70 pp.},
         url={https://doi.org/10.1016/j.aim.2023.108892},
}

\bib{FK}{article}{
      author={Frenkel, Igor~B.},
      author={Khovanov, Mikhail~G.},
       title={Canonical bases in tensor products and graphical calculus for
  {$U_q({\germ s}{\germ l}_2)$}},
        date={1997},
        ISSN={0012-7094},
     journal={Duke Math. J.},
      volume={87},
      number={3},
       pages={409\ndash 480},
         url={https://doi.org/10.1215/S0012-7094-97-08715-9},
}

\bib{Haz}{article}{
      author={Hazi, Amit},
       title={Existence and rotatability of the two-colored {J}ones-{W}enzl
  projector},
        date={2024},
        ISSN={0024-6093},
     journal={Bull. Lond. Math. Soc.},
      volume={56},
      number={3},
       pages={1095\ndash 1113},
         url={https://doi.org/10.1112/blms.12983},
}

\bib{Jones2}{article}{
      author={Jones, V. F.~R.},
       title={Index for subfactors},
        date={1983},
        ISSN={0020-9910},
     journal={Invent. Math.},
      volume={72},
      number={1},
       pages={1\ndash 25},
         url={https://doi.org/10.1007/BF01389127},
}

\bib{Jones1}{article}{
      author={Jones, V. F.~R.},
       title={A new knot polynomial and von {N}eumann algebras},
        date={1986},
        ISSN={0002-9920},
     journal={Notices Amer. Math. Soc.},
      volume={33},
      number={2},
       pages={219\ndash 225},
}

\bib{Kauf}{incollection}{
      author={Kauffman, Louis~H.},
       title={Statistical mechanics and the {J}ones polynomial},
        date={1988},
   booktitle={Braids ({S}anta {C}ruz, {CA}, 1986)},
      series={Contemp. Math.},
      volume={78},
   publisher={Amer. Math. Soc., Providence, RI},
       pages={263\ndash 297},
         url={https://doi.org/10.1090/conm/078/975085},
}

\bib{KauffBook}{book}{
      author={Kauffman, Louis~H.},
       title={Knots and physics},
      series={Series on Knots and Everything},
   publisher={World Scientific Publishing Co., Inc., River Edge, NJ},
        date={1991},
      volume={1},
        ISBN={981-02-0343-8; 981-02-0344-6},
         url={https://doi.org/10.1142/9789812796226},
}

\bib{KaufLins}{book}{
      author={Kauffman, Louis~H.},
      author={Lins, S\'{o}stenes~L.},
       title={Temperley-{L}ieb recoupling theory and invariants of
  {$3$}-manifolds},
      series={Annals of Mathematics Studies},
   publisher={Princeton University Press, Princeton, NJ},
        date={1994},
      volume={134},
        ISBN={0-691-03640-3},
         url={https://doi.org/10.1515/9781400882533},
}

\bib{KL}{article}{
      author={Khovanov, Mikhail},
      author={Laugwitz, Robert},
       title={Planar diagrammatics of self-adjoint functors and recognizable
  tree series},
        date={2023},
        ISSN={1558-8599},
     journal={Pure Appl. Math. Q.},
      volume={19},
      number={5},
       pages={2409\ndash 2499},
         url={https://doi.org/10.4310/pamq.2023.v19.n5.a4},
}

\bib{Krause}{article}{
      author={Krause, Henning},
       title={Krull-{S}chmidt categories and projective covers},
        date={2015},
        ISSN={0723-0869},
     journal={Expo. Math.},
      volume={33},
      number={4},
       pages={535\ndash 549},
         url={https://doi.org/10.1016/j.exmath.2015.10.001},
}

\bib{mason-handscome}{book}{
      author={Mason, J.C.},
      author={Handscomb, D.C.},
       title={Chebyshev polynomials},
   publisher={CRC Press},
        date={2002},
        ISBN={9781420036114},
         url={https://books.google.co.uk/books?id=g1DMBQAAQBAJ},
}

\bib{Mor}{inproceedings}{
      author={Morrison, Scott},
       title={A formula for the {J}ones-{W}enzl projections},
        date={2017},
   booktitle={Proceedings of the 2014 {M}aui and 2015 {Q}inhuangdao conferences
  in honour of {V}aughan {F}. {R}. {J}ones' 60th birthday},
      series={Proc. Centre Math. Appl. Austral. Nat. Univ.},
      volume={46},
   publisher={Austral. Nat. Univ., Canberra},
       pages={367\ndash 378},
}

\bib{TL}{article}{
      author={Temperley, H. N.~V.},
      author={Lieb, E.~H.},
       title={Relations between the ``percolation'' and ``colouring'' problem
  and other graph-theoretical problems associated with regular planar lattices:
  some exact results for the ``percolation'' problem},
        date={1971},
        ISSN={0962-8444},
     journal={Proc. Roy. Soc. London Ser. A},
      volume={322},
      number={1549},
       pages={251\ndash 280},
         url={https://doi.org/10.1098/rspa.1971.0067},
}

\bib{Wen}{article}{
      author={Wenzl, Hans},
       title={On sequences of projections},
        date={1987},
        ISSN={0706-1994},
     journal={C. R. Math. Rep. Acad. Sci. Canada},
      volume={9},
      number={1},
       pages={5\ndash 9},
}

\end{biblist}
\end{bibdiv}
\bibliographystyle{amsrefs}

\end{document}